\documentclass[reqno,11pt]{amsart}

\usepackage{epsf,color}
\usepackage{epsfig}
\usepackage{amsmath}
\usepackage{amssymb}
\usepackage[all]{xy}
\usepackage{mathrsfs}

\numberwithin{equation}{section}

\newcommand{\bC}{{\mathbb C}}

\newcommand{\bL}{{\mathbb L}}
\newcommand{\bP}{{\mathbb P}}
\newcommand{\bQ}{{\mathbb Q}}
\newcommand{\bR}{{\mathbb R}}
\newcommand{\bZ}{{\mathbb Z}}
\newcommand{\cA}{{\mathcal A}}
\newcommand{\cC}{{\mathcal C}}
\newcommand{\cE}{{\mathcal E}}
\newcommand{\cF}{{\mathcal F}}
\newcommand{\cG}{{\mathcal G}}
\newcommand{\cI}{{\mathcal I}}
\newcommand{\cL}{{\mathcal L}}
\newcommand{\cP}{{\mathcal P}}
\newcommand{\cO}{{\mathcal O}}
\newcommand{\cS}{{\mathcal S}}
\newcommand{\cT}{{\mathcal T}}
\newcommand{\cY}{{\mathcal Y}}
\newcommand{\cU}{{\mathcal U}}
\newcommand{\cX}{{\mathcal X}}

\newcommand{\sA}{{\mathscr A}}
\newcommand{\sB}{{\mathscr B}}
\newcommand{\sC}{{\mathscr C}}

\newcommand{\sF}{{\mathscr F}}
\newcommand{\scrA}{{\mathscr A}}
\newcommand{\scrB}{{\mathscr B}}
\newcommand{\scrC}{{\mathscr C}}
\newcommand{\scrD}{{\mathscr D}}
\newcommand{\scrF}{{\mathscr F}}
\newcommand{\scrQ}{{\mathscr Q}}
\newcommand{\scrX}{{\mathscr X}}
\newcommand{\scrY}{{\mathscr Y}}
\newcommand{\scrM}{{\mathscr M}}
\newcommand{\scrT}{{\mathscr T}}

\newcommand{\bfK}{\mathbf{K}}
\newcommand{\bfL}{\mathbf{L}}

\newcommand{\Pic}{\operatorname{Pic}}
\newcommand{\Ob}{\operatorname{Ob}}
\newcommand{\End}{\operatorname{End}}
\newcommand{\Hom}{\operatorname{Hom}}
\newcommand{\id}{\operatorname{id}}
\newcommand{\rk}{\operatorname{rk}}

\newcommand{\Coh}{\operatorname{Coh}}
\newcommand{\Ext}{\operatorname{Ext}}
\newcommand{\Tw}{\operatorname{Tw}}
\newcommand{\Symp}{\operatorname{Symp}}
\newcommand{\Ham}{\operatorname{Ham}}

\newcommand{\pt}{\operatorname{pt}}

\def\co{\colon\thinspace}

\newtheorem{thm}{Theorem}[section]
\newtheorem{cor}[thm]{Corollary}
\newtheorem{lem}[thm]{Lemma}
\newtheorem{prop}[thm]{Proposition}
\newtheorem{defin}[thm]{Definition}
\newtheorem*{axi}{Axiom of Quilted Floer Theory}
\newtheorem{Example}[thm]{Example}

\theoremstyle{remark}
\newtheorem{rem}[thm]{Remark}
\newtheorem{rems}[thm]{Remarks}

\newcommand{\noproof}{
\begin{flushright}
\qedsymbol
\end{flushright}}

\title[Mirror symmetry for the four-torus]{Homological mirror symmetry for the four-torus}

\author[M.~Abouzaid, I.~Smith]{Mohammed Abouzaid, \ Ivan Smith} \date{\today}
\thanks{This work was conducted during the time the first author was supported by a Clay Research Fellowship.  The second author is partially supported by grant ERC-2007-StG-205349 from the European Research Council. }

\begin{document}

\begin{abstract}
We use the quilt formalism of Mau-Wehrheim-Woodward to give a sufficient condition for a finite collection of Lagrangian submanifolds to split-generate the Fukaya category, and deduce homological mirror symmetry for the standard 4-torus.  As an application, we study Lagrangian genus two surfaces $\Sigma_2 \subset T^4$ of Maslov class zero, deriving numerical restrictions on the intersections of $\Sigma_2$ with linear Lagrangian 2-tori in $T^4$.
\end{abstract}

\maketitle

\maketitle
\setcounter{tocdepth}{1}
\tableofcontents

\section{Introduction}

Despite being the focus of a great deal of attention, Kontsevich's homological mirror symmetry conjecture \cite{Kontsevich} has been fully proved in only a handful of cases.   In the original Calabi-Yau setting, the elliptic curve was treated by Polishchuk and Zaslow \cite{PZ,P2}, whilst Seidel proved the conjecture for the quartic K3 surface \cite{Seidel:HMSquartic}. 
Substantial but partial results are known for a wide class of abelian varieties, from work of Fukaya \cite{Fukaya} and Kontsevich-Soibelman \cite{KS}.  In each of the last two studies, explicit embeddings of (subcategories of) the derived category into the Fukaya category were constructed, but it was not obvious that these embeddings actually induced equivalences.   An equivalence of superconformal field theories was separately established for flat tori by Kapustin and Orlov \cite{KO1}.

The step required to complete these arguments to a proof of mirror symmetry is therefore to show that the image subcategories generate the entire Fukaya category.  Nadler's paper \cite{nadler}, which concerns a certain Fukaya category of Lagrangian submanifolds of the cotangent bundle,  has highlighted the importance of the notion of \emph{resolutions of the diagonal} in proofs of such a result.  The idea is that the diagonal (as a Lagrangian in the square $M\times M$ of the given symplectic manifold $M$) represents the identity functor of the entire Fukaya category, so it suffices to show that the diagonal can be decomposed, in an appropriate sense, using products of Lagrangians that one understands.  Beyond the technical problem of formalising the connection between Floer theory on a symplectic manifold and on its square, this approach reduces the study of the Fukaya category of a symplectic manifold $M$ -- which might be teeming with countless unseen Lagrangians -- to the study of an explicit, often finite collection of Lagrangians in $M\times M$.  In algebraic geometry the corresponding circle of ideas is well-known, going back to Beilinson \cite{Beilinson}, and underlies Seidel's split-generation result for Fukaya categories of Lefschetz fibrations by vanishing cycles \cite[Remark 18.28]{Seidel:FCPLT}.

The main technical observation of this paper is that the  theory of pseudo-holomorphic quilts, currently under development by  Mau, Wehrheim and Woodward, not only establishes a functor relating the Fukaya category of a product to those of its factors, but moreover, using additional input from homological algebra, reduces the necessary computations on the product manifold to ones which can be performed on the original space.  The resulting generative criterion Theorem \ref{thm:generation}, applied to the standard four-torus, completes the proof of homological mirror symmetry in this case.  Amusingly,  the argument relies only on the proof of Homological Mirror Symmetry for the elliptic curve; we need not appeal to the deeper work of Fukaya \cite{Fukaya} or of Kontsevich-Soibelman \cite{KS}, nor perform any (serious) Floer theoretic computation in $T^4$.

To give a precise statement of the result we prove, let $\Lambda_{\bR}$ denote the Novikov field
\begin{equation} \label{eqn:novikov}
\Lambda_{\bR} = \left\{ \sum_{i\in \bZ} a_i q^{t_i} \ \big|  \  a_i\in\bC, a_i = 0 \ \textrm{for} \  i\ll 0, t_i\in\bR, t_i \to \infty \right\}.
\end{equation}
This is an algebraically closed field of characteristic zero.  Let $E$ denote the Tate elliptic curve, namely the projective algebraic variety  over $\Lambda_{\bR}$ with ring of functions
\[ \Lambda_{\bR}[x,y] \big/ \{y^2 + xy = x^3 + a_4(q) x + a_6(q)\}, \] for series $a_4, a_6 \in \bZ[[q]]$ defined by:
\[
s_k(q) = \sum_{m\geq 1} \frac{m^kq^m}{1-q^m}; \ a_4(q) =  -5s_3(q);  \ a_6(q) = (-5s_3(q)-7s_5(q))/12.
\]
For a projective variety $X$, we write $D^b(X)$ for its bounded derived category of coherent sheaves.
On the symplectic side, let $(T^4,\omega_{std})$ denote the four-torus $\bR^4/\bZ^4$ with its standard symplectic structure $\sum_j dx_j \wedge dy_j$, and write $D^{\pi}\scrF(T^4)$ for its split-closed derived Fukaya category (background on $A_{\infty}$-categories is given in the following section).

\begin{thm} \label{thm:hms}
There is an equivalence of triangulated categories (defined over the Novikov field) $D^{\pi}\sF(T^4,\omega_{std}) \simeq D^b(E\times E)$.
\end{thm}

\begin{rem}
As noted above, the proof of the split-generation criterion Theorem \ref{thm:generation} uses the Mau-Wehrheim-Woodward formalism which constructs $A_{\infty}$-functors between Fukaya categories from counts of quilted Riemann surfaces.   That theory is still under construction \cite{MWW,W:talk}, but -- by design -- is easily formalised and axiomatised; indeed that is one of the theory's most important and satisfactory features.  We give a condensed overview of quilt theory in situations -- like that relevant to Theorem \ref{thm:hms} -- in which bubbling is not an issue; the properties that we require are subsumed in the \emph{Axiom of Quilted Floer Theory} of Section \ref{Subsec:BackgroundQuilts}. \emph{These expected formal properties are assumed throughout this paper}.  
\end{rem}

\begin{rem}
During the course of the proof, we will see that $D^{\pi}\sF(T^2)$ is split-generated by a meridional and a longitudinal circle on the two-torus, whilst $D^{\pi}\sF(T^4)$ is split-generated by the four two-tori obtained by taking pairwise products of these; this is essentially a ``K\"unneth  theorem" for the Fukaya category.
\end{rem}

The strategy of proof of Theorem \ref{thm:hms} is to relate the Fukaya category $\sF(X\times Y)$ of the product to the category of functors $\Hom(\sF(X),\sF(Y))$.  Modulo (serious) technical restrictions, the argument should imply homological mirror symmetry for $X\times Y$ whenever it is known for $X$ and $Y$ themselves.  The technical restrictions are largely foundational in nature.  By definition, the objects of the Fukaya category are Lagrangian submanifolds of Maslov class zero (decorated with some additional ``brane" structure, cf. Section \ref{Subsec:BackgroundFukaya}).  In dimension four, any such submanifold is unobstructed: it bounds no holomorphic disc for generic almost complex structures, just for index reasons\footnote{This idea was used by Seidel in setting up and proving mirror symmetry for the quartic surface in \cite{Seidel:HMSquartic}.}.  In higher dimensions, the construction of the Fukaya category relies, in general, on the delicate obstruction theory of Fukaya-Oh-Ohta-Ono \cite{FO3}, see also Joyce's \cite{Joyce}. Whilst that theory is now coming into final shape, the issues that obstruction chains might raise in the quilted setting have not yet been addressed.  That is the reason Theorem \ref{thm:hms} is restricted to four-dimensions.    For an aspherical symplectic manifold $X$,  denote by $\sF_{so}(X)$ the ``strictly unobstructed" Fukaya category whose objects are (decorated) pairs $(L,J)$ comprising a Maslov zero Lagrangian submanifold and an almost complex structure $J$ for which $L$ bounds no $J$-holomorphic discs. The technology of the current paper carries over rather straightforwardly to prove:

\begin{cor} \label{cor:highdimension}
For any integer $k$, there is an equivalence of categories (defined over the Novikov field) $D^{\pi}\sF_{so}(T^{2k}, \omega_{std}) \simeq D^b\Coh(E^k)$.
\end{cor}

When $k$ is a power of $2$, there is nothing beyond what we prove in this paper together with induction.  Otherwise, the relevant homological algebra is slightly more involved,  requiring the use of categories of functors rather than simply of endofunctors.  More generally, if one were willing to flesh out the arguments presented by Kontsevich and Soibelman in \cite{KS}, one could conclude a similar result for a wide range of abelian varieties.  However, the category $\sF_{so}(X)$ is of at best marginal interest from the point of view of symplectic topology, so  we will only discuss four-manifolds in the body of this paper, deferring some further speculative remarks to the  Appendix.

In fact, proving mirror symmetry was in some sense a subsidiary of our main intention, which was to explore its consequences for symplectic topology.  In this vein, our principal result is:

\begin{thm} \label{thm:classify}
Let $\Sigma_2\subset (T^4,\omega_{std})$ be a Lagrangian genus two surface of Maslov class zero.  Then $\Sigma_2$ is Floer cohomologically indistinguishable from the Lagrange surgery of some pair of linear Lagrangian tori meeting transversely once.
\end{thm}

Although this is a rather specialized result, we should point out that it seems currently inaccessible without mirror symmetry: Lagrangian isotopy or uniqueness theorems in four dimensions typically rely on constructions of holomorphic foliations which do not exist, generically, on the torus, and our proof relies crucially on  results of Mukai and Orlov on sheaves on abelian varieties.
Despite its rather abstract formulation, Theorem \ref{thm:classify} has direct implications for intersection properties of Lagrangian submanifolds of $T^4$; it imposes numerical restrictions on such intersections reminiscent of those implied by the Arnol'd conjecture.

\begin{cor} \label{cor:intersect}
Let $\Sigma_2 \subset (T^4,\omega_{std})$ be a Lagrangian genus 2 surface of Maslov class zero. There are at least two Lagrangian tori $L_0,L_1\subset T^4$ (in rationally independent homology classes) with the property that any surface Hamiltonian isotopic to $\Sigma_2$ and meeting $L_{i}$ transversely does so in at least 3 points. 
\end{cor}

\begin{rem}
For  $\Sigma_2$ arising as a Lagrange surgery of a pair of transverse linear Lagrangian tori, one can isotope $\Sigma_2$ through Lagrangian non-Hamiltonian isotopies to meet any linear Lagrangian torus at most once.
\end{rem}

In the setting of Corollary \ref{cor:highdimension}, a Maslov zero Lagrangian submanifold $L \subset T^{2k}$ must have non-vanishing Floer cohomology with at least one $(F,\xi)$ comprising a torus fibre $F$ of a Lagrangian fibration $T^{2k} \rightarrow T^k$ equipped with a local system $\xi\rightarrow F$, or $L$ must bound a $J$-holomorphic disc for every compatible almost complex structure $J$.  This is just the mirror of the fact that no complex of sheaves on an algebraic variety can have vanishing $\Ext$'s with the structure sheaves of all closed points; but it seems far from obvious by direct symplectic arguments.  In another direction, Kapustin and Orlov \cite{KO2} argued that the Fukaya category of a high-dimensional torus did not seem ``large enough" (from the point of view of $K$-theory) to be mirror to the derived category of sheaves of an abelian variety, and suggested that one should augment the Fukaya category with certain co-isotropic branes; but Corollary \ref{cor:highdimension} indicates that, after passing to split-closures, these are in fact not necessary.



\subsection*{Acknowledgments}  Thanks to Daniel Huybrechts, Dima Orlov, Paul Seidel and Nick Shepherd-Barron for helpful suggestions and correspondence, as well as to Sikimeti Mau, Katrin Wehrheim and Chris Woodward for sharing with us a preliminary draft of their paper \cite{MWW}.   Detailed comments on our own preliminary draft from many of the same people have saved us from numerous (additional) errors: we are especially indebted to Dima Orlov.   Extensive and constructive comments from the anonymous referees have also greatly improved the exposition.


\section{$A_{\infty}$-algebra\label{Subsec:BackgroundAlgebra}}

We begin by collecting some basic facts about $A_{\infty}$-categories. These results were proved by various authors, starting with Kadeishvili \cite{Kadeishvili}. Any discussion of $A_{\infty}$-categories involves fixing sign conventions,  of which there are several; for consistency, we have 
chosen to cite all results from Seidel's book \cite{Seidel:FCPLT}, whose 
conventions 
we borrow.  The reader entirely unfamiliar with this material will 
probably find our treatment too cursory, and is invited to consult Part I of Seidel's book for a more leisurely development, for proofs,  and for the appropriate references to the original papers where the results were proved.

\subsection{$A_{\infty}$-categories and functors}

Fix an arbitrary field $k$.  A non-unital $A_{\infty}$-category $\scrA$ over $k$ comprises: a set of objects Ob$\,\scrA$; for each $X_0, X_1 \in$ Ob$\,\scrA$ a graded $k$-vector space $hom_{\scrA}(X_0, X_1)$; and $k$-linear composition maps, for $d \geq 1$,
\[
\mu_{\scrA}^d: hom_{\scrA}(X_{d-1},X_d) \otimes \cdots \otimes hom_{\scrA}(X_0,X_1) \longrightarrow hom_{\scrA}(X_0,X_d)[2-d]
\]
of degree $2-d$ (the notation $[l]$ refers to \emph{downward} shift by $l \in \bZ$).  The maps $\{\mu^d\}$ satisfy a hierarchy of quadratic equations
\[
\sum_{m,n} (-1)^{\maltese_n} \mu_{\scrA}^{d-m+1} (a_d,\ldots,a_{n+m+1}, \mu_{\scrA}(a_{n+m},\ldots,a_{n+1}),a_n \ldots, a_1) = 0
\]
with $\maltese_n = \sum_{j=1}^n |a_j|-n$ and where the sum runs over all possible compositions: $1 \leq m \leq d$, $0\leq n \leq d-m$.  The equations imply in particular that $hom_{\scrA}(X_0,X_1)$ is a cochain complex with differential $\mu^1_{\scrA}$; the cohomological category $H(\scrA)$ has the same objects as $\scrA$ but morphism groups are the cohomologies of these cochain complexes.  This has an associative composition 
\[
[a_2]\cdot[a_1] = (-1)^{|a_1|}[\mu_{\scrA}^2(a_2,a_1)]
\]
but the chain-level composition $\mu_{\scrA}^2$ on $\scrA$ itself is only associative up to homotopy.  (Thus $\scrA$  is not strictly a category.)  Moreover, the higher-order compositions $\mu_{\scrA}^d$ are \emph{not} chain maps, and do not descend to cohomology.    With appropriate sign conventions, a $dg$-category is just the special case in which $\mu_{\scrA}^d = 0$ for all $d>2$.

A non-unital $A_{\infty}$-functor $\cF: \scrA \rightarrow \scrB$ between non-unital $A_{\infty}$-categories $\scrA$ and $\scrB$ comprises a map $\cF$: Ob$\,\scrA \rightarrow$ Ob$\,\scrB$, and multilinear maps for $d\geq 1$
\[
\cF^d: hom_{\scrA}(X_{d-1}, X_d) \otimes \cdots \otimes hom_{\scrA}(X_0,X_1) \rightarrow hom_{\scrB}(\cF X_0, \cF X_1)[1-d]
\]
now satisfying the polynomial equations
\begin{multline*}
\sum_{r} \sum_{s_1+\cdots+s_r=d} \mu_{\sB}^r(\cF^{s_r}(a_d,\ldots,a_{d-s_r+1}),\ldots, \cF^{s_1}(a_{s_1},\ldots, a_{1})) \\
\quad  = \ \sum_{m,n} (-1)^{\maltese_n} \cF^{d-m+1}(a_d,\ldots, a_{n+m+1}, \mu_{\sA}^m (a_{n+m},\ldots, a_{n+1}), a_{n},\ldots a_1)
\end{multline*}
Any such defines a functor $H(\cF):H(\scrA) \rightarrow H(\scrB)$ which takes $[a] \mapsto [\cF^1(a)]$; if $H(\cF)$ is an isomorphism, respectively full and faithful, we say $\cF$ is a quasi-isomorphism, respectively cohomologically full and faithful.  The collection of $A_{\infty}$-functors from $\scrA$ to $\scrB$ themselves form the objects of a non-unital $A_{\infty}$-category $\scrQ=nu$-$fun(\scrA,\scrB)$, whose morphism groups are groups of $A_{\infty}$-natural transformations.   Concretely, given two functors $\cF$ and $\cG$, one defines a \emph{pre-natural transformation} $T$ to be a sequence $(T^0, T^1, \ldots)$, with $T^d$ a collection of maps
\begin{equation*} hom_{\scrA}(X_{d-1}, X_d) \otimes \cdots \otimes hom_{\scrA}(X_0,X_1) \to hom_{\scrB}(\cF X_0, \cG X_d)[2-d]  \end{equation*}
for all sequences $(X_0, \ldots, X_d)$ of objects in $\scrA$.  The vector space generated by all pre-natural transformations is by definition the space of morphisms between $\cF$ and $\cG$ in $ \scrQ$.  The formula for the differential on $ hom_{\scrQ}(\cF, \cG)$ can be found, for example, as Equation (1.8) in \cite{Seidel:FCPLT}.

The vector space  $ hom_{\scrQ}(\cF, \cG)$ admits a decreasing length filtration with  $ hom_{\scrQ}^{r}(\cF, \cG)$ consisting of all pre-natural transformations for which $T^0 = T^1 = \cdots = T^{r-1} = 0$.  The associated spectral sequence has as its first page
\begin{multline} \label{eq:spectral_sequence_length}  E^{r,s}_1 = \prod_{X_0, \ldots, X_r} \Hom^{s}\big( hom_{H(\scrA)}(X_{r-1},X_r) \otimes \cdots \otimes  hom_{H(\scrA)}(X_{0},X_1), \\
hom_{H(\scrB)}(\cF(X_0), \cG(X_r)) \big) \end{multline}

We should emphasise that there are in general more $A_{\infty}$-functors between $dg$-categories than $dg$-functors even at the level of homology, so when the higher order operations $(\mu_{\scrA}^d)_{d\geq 2}$ vanish, there is content to regarding $\scrA$ as an $A_{\infty}$-category.    In the other direction, Kadeishvili \cite{Kadeishvili} showed:

\begin{lem}[Homological Perturbation Lemma]\label{lem:HPL}
Any $A_{\infty}$-category is quasi-isomorphic to an $A_{\infty}$-structure on its cohomology 
\[
 (\scrA, \{\mu^d_{\scrA}\}_{d\geq 1}) \quad \simeq \quad 
(H(\scrA), \{\mu^d_{H(\scrA)}\}_{d\geq 2})\]
 with $\mu^1_{H(\scrA)}=0$.
\end{lem}

If $\scrA \simeq (H(\scrA), \{\mu^d_{H\scrA}\})$  is quasi-isomorphic to the trivial $A_{\infty}$-structure on its cohomology, namely to the structure with $\mu^d_{H\scrA} \equiv 0$ for $d\geq 3$, we say that $\scrA$ is \emph{formal}.

\subsection{$A_{\infty}$-modules}
An $A_{\infty}$-category $\scrA$ has a category $\mod$-$(\scrA)$ of right $A_{\infty}$-modules, which abstractly is  the category of functors $nu$-$fun(\scrA^{opp},Ch)$ from (the opposite of) $\scrA$ to the $dg$-category of chain complexes of graded $k$-vector spaces.  Concretely, an $A_{\infty}$-module $\scrM$ associates to any $X\in$ Ob$\,\scrA$ a graded vector space $\scrM(X)$, and there are maps for $d\geq 1$
\[
\mu_{\scrM}^d: \scrM(X_{d-1}) \otimes hom_{\scrA}(X_{d-2},X_{d-1}) \otimes \cdots \otimes hom_{\scrA}(X_0,X_1) \rightarrow \scrM(X_0)[2-d].
\]
The $A_{\infty}$-functor equations imply in particular that $\mu_{\scrM}^1$ is the differential on the chain complex  $\scrM(X_0)$.  In the sequel, we shall make particular use of the  $A_{\infty}$-version of the Yoneda Lemma \cite[Lemma 2.12]{Seidel:FCPLT}.  To state this, note that for any $Y \in$ Ob$\,\scrA$, there is an associated module $\scrY$ defined by
\[
\scrY(X) = hom_{\scrA}(X,Y); \quad \mu_{\scrY}^d = \mu_{\scrA}^d. 
\]
The association $Y \mapsto \scrY$ extends to a canonical non-unital $A_{\infty}$-functor $\scrA \rightarrow mod$-$(\scrA)$, the Yoneda embedding.  There is a dual Yoneda embedding into the same category of modules
\[
Y \mapsto \scrY^{\vee}; \qquad \scrY^{\vee}(X) = (hom_{\scrA}(Y,X))^{\vee}
\]
with the module structure given by
\begin{equation*} \mu_{\scrY^{\vee}}^d( \phi, a_d, \ldots, a_1)(b) = \phi( \mu_{\scrA}^d(b, a_d, \ldots, a_1) ). \end{equation*}

The proof of the next result uses the fact that the spectral sequence \eqref{eq:spectral_sequence_length} collapses at the second page to the column $r=0$ because of the acyclicity of the bar resolution:
\begin{lem}[Yoneda Lemma] \label{lem:Yoneda}
The natural map  $\lambda: \scrM(Y) \rightarrow hom_{\scrQ}(\scrY,\scrM)$ taking
\[
\lambda(c)^d(b,a_{d-1},\ldots,a_1) = \mu_{\scrM}^{d+1}(c,b,a_{d-1},\ldots,a_1)
\]
is a quasi-isomorphism.  The association $Y \mapsto \scrY$ defines a cohomologically full and faithful functor $\scrA \rightarrow \scrQ$.
\end{lem}

The category of $A_{\infty}$-modules is naturally a \emph{triangulated} $A_{\infty}$-category:  morphisms have cones.  More precisely, if $c\in hom_{\scrA}(Y_0,Y_1)$ is a degree zero cocycle, $\mu^1_{\scrA}(c)=0$, there is an $A_{\infty}$-module $Cone(c)$ defined by
\begin{equation}\label{eqn:cone}
Cone(c)(X) = hom_{\scrA}(X,Y_0)[1] \oplus hom_{\scrA}(X,Y_1)
\end{equation}
and with operations $\mu^d_{Cone(c)}((b_0,b_1),a_{d-1},\ldots,a_1)$ given by the pair of terms
\[
\left(\mu_{\scrA}^d(b_0,a_{d-1},\ldots,a_1), \mu_{\scrA}^d(b_1,,a_{d-1},\ldots,a_1)+ \mu_{\scrA}^{d+1}(c,b_0,a_{d-1},\ldots,a_1) \right).
\]
One can generalise the mapping cone construction and consider \emph{twisted complexes} in $\scrA$: a twisted complex is a pair $(X,\delta_X)$ where $X$ is a formal direct sum
\begin{equation} \label{eq:twisted_complex_vector_space}
X = \oplus_{i\in I} V^i \otimes X^i
\end{equation}
with $\{X^i\} \in$ Ob$\,\scrA$ and $V^i$ finite-dimensional graded $k$-vector spaces -- i.e. $X$ is an object of the ``additive enlargement" $\Sigma\scrA$ --  and where $\delta_X \in hom^1_{\Sigma\scrA}(X,X)$ is a matrix of differentials
\begin{equation} \label{eq:twisted_complex_differential}
\delta_X = (\delta_X^{ji});\qquad \delta_X^{ji} = \sum_k \phi^{jik} \otimes x^{jik}
\end{equation}
with $\phi^{jik} \in Hom_k(V^i,V^j)$, $x^{jik} \in hom_{\scrA}(X^i,X^j)$ and having total degree $|\phi^{jik}| + |x^{jik}| = 1$.  The differential $\delta_X$ should satisfy the two properties
\begin{itemize}
\item $\delta_X$ is strictly lower-triangular with respect to some filtration of $X$;
\item $\sum_{r=1}^{\infty} \mu^r_{\Sigma\scrA}(\delta_X,\ldots,\delta_X)=0$.
\end{itemize}
Twisted complexes themselves form the objects of a non-unital $A_{\infty}$-category $Tw(\scrA)$, which has the property that all morphisms can be completed with cones to sit in exact triangles.  A basic  example of a twisted complex is that obtained simply by taking the tensor product of an object $X$ by the vector space $k$ placed in some non-zero degree; in particular, the category $Tw(\scrA)$ has a \emph{shift functor}, which practically has the effect of shifting all degrees of all morphism groups downwards by one.

A particularly important class of mapping cones are those arising from \emph{twist functors}.  Given $Y \in$ Ob$\,\scrA$ and an $\scrA$-module $\scrM$, we define the twist $\scrT_Y\scrM$ as the module
\[
\scrT_Y\scrM(X) = \scrM(Y)\otimes hom_{\scrA}(X,Y)[1] \oplus \scrM(X)
\]
(with operations we shall not write out here).    The twist is the cone over the canonical evaluation morphism
\[
\scrM(Y) \otimes \scrY \rightarrow \scrM
\]
where $\scrY$ denotes the Yoneda image of $Y$.  For two objects $Y_0, Y_1 \in$ Ob$\, \scrA$, the essential feature of the twist is that it gives rise to a canonical exact triangle in $H(\scrA)$
\[
\cdots \rightarrow Hom_{H(\scrA)}(Y_0,Y_1) \otimes Y_0 \rightarrow Y_1 \rightarrow T_{Y_0}(Y_1) \stackrel{[1]}{\rightarrow} \cdots
\]
(where $T_{Y_0}(Y_1)$ is any object whose Yoneda image is $\scrT_{Y_0}(\scrY_1)$).  Twist functors have played a critical role in relating algebraic properties of Fukaya categories and geometric properties of the underlying symplectic manifold, as briefly indicated below (Proposition \ref{prop:seidel}).

\subsection{Unitality and projection functors} 
The $A_{\infty}$-categories we shall consider are \emph{cohomologically unital}, meaning the categories $H(\scrA)$ are unital (objects have identity morphisms, so $H(\scrA)$ is a graded linear category in the usual sense).  In fact, if one makes careful choices in constructing the Fukaya category, its objects are equipped with lifts of the cohomological units satisfying
\begin{align*}  \mu^{2}(e, e) & = e  \\
\mu^{d}(e, \ldots, e) & = 0  \textrm{ if }d>2. \end{align*}
From a more formal point of view, we note that the existence of an element satisfying these conditions for an object $Y \in \scrB$ is equivalent to the existence of an $A_{\infty}$-functor
\begin{equation} \label{eq:unit_functor}  \cU_{Y} \co k \to \scrB \end{equation}
with vanishing higher order terms.  Here, we think of $k$ as a category with one object whose endomorphism algebra is the ground field $k$; the functor takes this unique object to $Y$, and its linear term maps $1 \in k$ to the unit of $Y$.

Starting with any finite dimensional co-chain complex $C^*$ and such a linear functor with target $Y$, we may naturally define a twisted complex
\begin{equation} \label{eq:tensor_product_chain_cmplex} C^{*} \otimes Y  \end{equation}
In the notation of Equations \eqref{eq:twisted_complex_vector_space} and \eqref{eq:twisted_complex_differential}, the vector spaces $V^i$ are the graded components of $C^*$, and the differential $\delta^{ji}$ is only non-vanishing if $j=i+1$, in which case it is given by the tensor product of the differential on $C^*$ with the identity on $Y$:
\begin{equation*}  \partial_{C^*} \otimes e. \end{equation*}

The existence of such units also allows us to define certain projection functors which in the general cohomologically unital case only make sense after passing to modules.
\begin{defin} \label{defin:projection-functor}
For each pair of objects $(X_-,Y_+)$ of $\scrA$ and $\scrB$ with $Y_+$ the image of a linear functor with source $k$, we define a \emph{projection} functor
\begin{equation} \cI(X_-,Y_+): \scrA \rightarrow Tw(\scrB) \end{equation}
which acts on objects by
\[ X \longrightarrow hom_{\scrA}(X_-, X) \otimes Y_{+} \] 
and on morphisms by
\begin{align} 
\notag hom_{\scrA}(X_{d-1},X_{d}) \otimes  \cdots \otimes hom_{\scrA}(X_{0}, X_{1}) &  \longrightarrow   \\   & \hspace{-2.4in}
\Hom\big(hom_{\scrA}(X_-, X_0), hom_{\scrA}(X_-, X_d)\big) \otimes hom_{\scrB}(Y_+, Y_+)[1-d] \\ \label{eq:formula_yoneda_functor-general}
a_{d} \otimes \cdots \otimes a_{1} &  \mapsto \mu^{d+1}_{\scrA} (a_{d}, \ldots, a_{1}, \cdot) \otimes e
\end{align}
\end{defin}

One may understand this construction from the abstract point of view by noting that the existence of the twisted complex $C^{*} \otimes Y$ is part of a higher \emph{tensor structure} on the category of $A_{\infty}$-categories.  As we shall not require the full power of such $(\infty,2)$-categorical machinery, and as the axiomatics of such a structure are rather delicate, we focus instead on a special situation which exploits the fact that the tensor product of an $A_{\infty}$-category $\cA$ and a $dg$-category $\cC$ may be easily defined by a formula analogous to \eqref{eq:formula_yoneda_functor-general}.

In particular, given  $A_{\infty}$-categories $\scrA$ and $\scrB$, $dg$-categories $\scrC$ and $\scrD$, and $A_{\infty}$-functors $\cF \co \scrA \to \scrD$, and $\cG \co \scrC \to \scrB$ such that $\cG$ has no higher order terms, one may define the tensor product of $\cF$ and $\cG$
\begin{equation*} \cF \otimes \cG \co \scrA \otimes \scrC \to \scrD \otimes \scrB \end{equation*}
whose higher order terms are obtained by applying the higher order terms of $\cF$.

If $\scrC$ is the category $k$ with one object, we have a canonical isomorphism $\scrA \otimes k \cong \scrA$, while if $\scrD$  is the category of chain complexes over $k$, we have a fully faithful embedding
\begin{equation} \label{eq:functor_tensor_complexes}  Ch \otimes \scrB  \to Tw(\scrB)  \end{equation} 
given by the construction of Equation \eqref{eq:tensor_product_chain_cmplex}.   In this language,  the functor $\cI(X_-,Y_+)$  is isomorphic to the composition of \eqref{eq:functor_tensor_complexes} with the tensor product of the Yoneda functor $\cX_-$ for $X_-$ and the functor $\cU_{Y_+} \co k \to \scrB$ from Equation \eqref{eq:unit_functor}:
\begin{equation*} \xymatrix{  
\scrA \ar[rr]^{\cI(X_-,Y_+)} \ar[dr]^{ \cX_-\otimes  \cU_{Y_+}}  & & Tw(\scrB) \\
& Ch \otimes   \scrB \ar[ur] & }
    \end{equation*} 

Given  a pair  of objects $(X_-,Y_+)$ and $(X_-', Y_+')$ such that $Y_+$ and $Y'_+$ are the images of $k$ under  linear functors, let us write $\cI = \cI(X_-,Y_+) $ and  $\cI' = \cI(X'_-,Y'_+) $. As a result of the previous discussion we expect that   $hom_{\scrQ}(\cI, \cI')$ is quasi-isomorphic to the tensor product of morphisms from $X_-$ to $X'_-$ with morphisms from  $Y_+$ to $Y'_+$.  To prove this, we consider the functors at the level of homological categories 
\begin{equation*} H^*\cI \textrm{ and }   H^*\cI' \co  H^*(\scrA) \to H^*(  Tw(\scrB) ) \end{equation*}
which are honest functors with no higher order terms.  There is a natural map
\begin{align}  \label{eq:a_infty_to_homological_functor}
H^* hom_{\scrQ}(\cI, \cI') & \to  hom_{H^*(\scrQ)}( H^*\cI,  H^*\cI') \\
[(T_0, T_1, \ldots) ] & \mapsto [T_0]
\end{align} 
where $H^*(\scrQ)$ is the category of cohomological functors in which morphisms consist of natural transformations.  The standard categorical Yoneda argument (rather than an $A_{\infty}$ version thereof) allows us to readily compute that
\begin{equation} hom_{H^*(\scrQ)}( H^*\cI,  H^*\cI')  \cong  hom_{H^*(\scrA)}( X_-,  X'_- )) \otimes hom_{H^*(\scrB)}(Y_+, Y'_+) . \end{equation}

\begin{lem}  \label{lem:yoneda_for_projection} 
Every natural transformation between $\cI$ and  $\cI'$ is detected at the level of homology, i.e. the map \eqref{eq:a_infty_to_homological_functor} is an isomorphism.  In particular,
\begin{equation}H^* hom_{\scrQ}( \cI,  \cI')  \cong  hom_{H^*(\scrA)}( X_-,  X'_- )) \otimes hom_{H^*(\scrB)}(Y_+, Y'_+) . \end{equation}
\end{lem}
\begin{proof}[Sketch of proof]
The proof is a minor generalisation of that of the Yoneda Lemma.  Namely, we consider the length filtration on $hom_{\scrQ}(\cI, \cI')$, and observe that Equation \eqref{eq:spectral_sequence_length} specialises, in this case to 
\begin{multline}E^{r,s}_1 =  \prod_{X_0, \ldots, X_r} \Hom^{s} \big( hom_{H(\scrA)}(X_{r-1},X_r) \otimes \cdots \otimes  hom_{H(\scrA)}(X_{0},X_1), \\
\Hom( hom_{H(\scrA)}(X_{-},X_0), hom_{H(\scrA)}(X'_-, X_r))  \otimes hom_{H(\scrB)  } (Y_+, Y'_+) \big) \end{multline}

Using adjunction, this can be more conveniently rewritten as
\begin{multline}E^{r,s}_1 =  \prod_{X_0, \ldots, X_r} \Hom^{s} \big( hom_{H(\scrA)}(X_{r-1},X_r) \otimes \cdots \otimes  hom_{H(\scrA)}(X_{0},X_1) \otimes \\
 hom_{H(\scrA)}(X_{-},X_0) , hom_{H(\scrA)}(X'_-, X_r) \big)  \otimes hom_{H(\scrB)  } (Y_+, Y'_+).\end{multline}
Note that this is the tensor product of the $E^1$ page of the spectral sequence computing $hom_{\scrQ}(  \scrX_-,   \scrX'_-  ) $ with the graded vector space $ hom_{H(\scrB)  } (Y_+, Y'_+)$.   As the differential $\partial_1^{r,s}$ involves only the product on homology, it is easy to check that it is given by the tensor product of the differential on the $E^1$ page for  $hom_{\scrQ}(  \scrX_-,   \scrX'_-  ) $ with the identity on  $ hom_{H(\scrB)  } (Y_+, Y'_+)$.  As the spectral sequence for $hom_{\scrQ}(  \scrX_-,   \scrX'_-  ) $ collapses at the second page to the  column $r=0$, we conclude the same result for  $hom_{\scrQ}(\cI(X_-,Y_+), \cI(X'_-,Y'_+))$.  Note that the column $r=0$ precisely consists of pre-natural transformations with non-zero $T^0$ term, i.e. ones which survive the projection to $hom_{H^*(\scrQ)}( H^*\cI,  H^*\cI')$.  
\end{proof}

\subsection{Idempotents and homological invariants}
The split-closed (also called idempotent-closed, or Karoubi-complete) derived category $D^{\pi}(\scrA)$ of $\scrA$ is obtained from $\scrA$ by splitting idempotent endomorphisms.  Instead of giving the details, we just point out \cite{Seidel:FCPLT} that an $A_{\infty}$-category is idempotent-closed if and only if its cohomological category $H(\scrA)$ has the same property, so one can view the passage from $Tw(\scrA)$ to $D^{\pi}(\scrA)$ as formally including objects which represent summands of endomorphism rings associated to cohomological idempotents.  If the smallest split-closed triangulated $A_{\infty}$-category containing a subcategory $\scrA' \subset \scrA$  is $D^{\pi}(\scrA)$, then we will say that $\scrA'$ split-generates $\scrA$.

To conclude the background in algebra, we mention two homological invariants of an $A_{\infty}$-category. The first is the K-theory, or rather the Grothendieck group 
\begin{equation} \label{eqn:K0}
K_0(\scrA) \ = \ \bZ\,\mathrm{Ob}\,Tw(\scrA) \big/ \langle [A]+[B]-[C] \rangle 
\end{equation}
where we impose a relation whenever $C$ is quasi-isomorphic to the mapping cone of a closed degree one morphism $A \rightarrow B$.  From the definition of twisted complexes, the $K_0$-group is actually generated by objects of $\scrA$ (by contrast its behaviour under passing to split-closure is rather wild in general).
Lastly, we also recall the definition of the Hochschild cohomology of an $A_{\infty}$-category $\scrA$. The most concise definition is to view  $HH^*(\scrA) = H(hom_{fun(\scrA,\scrA)}(\id,\id))$ as the morphisms in the $A_{\infty}$-category of endofunctors of $\scrA$ from the identity functor to itself.  More prosaically, $HH^*(\scrA)$ is computed by a chain complex $CC^*(\scrA)$ as follows.  A degree $r$ cochain is a sequence $(h^d)_{d\geq 0}$ of collections of linear maps
\[
h^d_{(X_1,\ldots,X_{d+1})}: \bigotimes_{i=d}^1 hom_{\scrA}(X_i,X_{i+1}) \rightarrow hom_{\scrA}(X_1,X_{d+1})[r-d]
\]
for each $(X_1,\ldots,X_{d+1})\in \Ob(\scrA)^{d+1}$.   The differential is defined by the usual sum over possible concatenations
\begin{equation}
\begin{aligned}
(\partial h)^d & (a_d,\ldots, a_1) = \\
& \sum_{i+j<d+1} (-1)^{(r+1)\maltese_i} \mu_{\scrA}^{d+1-j}(a_d,\ldots,a_{i+j+1},h^j(a_{i+j},\ldots,a_{i+1}),a_{i},\ldots,a_1) \\
+ &  \sum_{i+j\leq d+1} (-1)^{\maltese_{i} +r +1}  h^{d+1-j} (a_d,\ldots,a_{i+j+1},\mu_{\scrA}^j(a_{i+j},\ldots,a_{i+1}),a_{i},\ldots,a_1).
\end{aligned}
\end{equation}

In particular, $h^{0}$ defines an endomorphism of every object of the category.  The above formula for the differential readily implies the next result:
\begin{lem} \label{lem:HHrestriction}
For any object $X$ of $\scrA$, the assignment
\begin{align*}
CC^{*}( \scrA) & \to hom_{\scrA} (X,X) \\
 (h^{d})_{d \geq 0} & \mapsto h^{0}_{X} 
\end{align*}
is a chain map. \noproof
\end{lem}

Classically, Hochschild cohomology arises in deformation theory.  Any (formal, i.e. ignoring convergence issues) deformation of the $A_{\infty}$-structure defines a  class in $HH^2(\scrA)$, so for instance if this is one-dimensional the category has a unique such deformation up to quasi-isomorphism.   We should also point out the following basic algebraic fact.

\begin{prop} \label{rem:hh-invariant}
Hochschild cohomology $HH^*(\scrA)$ is invariant under taking twisted complexes and under passing to idempotent completion.
\end{prop}

There seems to be no written account of this result in the setting of $A_{\infty}$-categories over a field, but the more general result for spectra is Theorem 4.12 of \cite{BM}.

\section{The Fukaya category\label{Subsec:BackgroundFukaya}}

Let $(M,\omega)$ be a closed symplectic manifold and suppose $2c_1(M)=0$. In ideal situations, the Fukaya category $\sF(M)$ is a triangulated $\bZ$-graded $A_{\infty}$-category, linear over the Novikov field $\Lambda_{\bR}$.  It has an associated (honest) triangulated category $D^{\pi}\sF(M)$, the split-closed  derived Fukaya category; one can also pass directly to cohomology, forgetting the $A_{\infty}$-structure, to obtain the  (quantum or \emph{Donaldson}) category $H(\sF(M))$. 
The objects of the Fukaya category are Lagrangian submanifolds which are decorated with additional data, the existence of which form a collection of strong constraints: 
\begin{itemize}
\item the submanifolds should have vanishing Maslov class, and be equipped with gradings \cite{Seidel:grading};
\item the submanifolds should be spin, or relatively spin relative to a fixed background class $b\in H^2(M;\bZ_2)$ (though not strictly necessary we moreover only consider orientable Lagrangians);
\item the Floer cohomology of the submanifolds should be unobstructed for some choice of bounding chains in the sense of \cite{FO3}.
\end{itemize}

Moreover, the Fukaya category should have the properties that
\begin{itemize}
\item Hamiltonian isotopic Lagrangian submanifolds define isomorphic objects of $\sF(M)$;
\item up to quasi-equivalence $\sF(M)$ is a symplectic invariant of $M$.
\end{itemize}
The final statement is intentionally vague: one expects a canonical map $\Symp^0(M)/\Ham(M) \rightarrow Auteq(D^{\pi}\sF(M))/\langle \bZ\rangle$, where $\Symp^0$ is a natural subgroup of symplectomorphisms which preserve the structure needed to grade the category, e.g. the homotopy class of trivialisation of $K_M^{\otimes 2}$, and on the right hand side we divide out by the shift functor. 
As indicated in Section \ref{Subsec:BackgroundAlgebra}, 
the construction of the derived category $D^{\pi}\sF(M)$ from $\sF(M)$ is a purely algebraic procedure.

\begin{rem} \label{rem:unitarity}
The Novikov field $\Lambda_{\bR}$ comprises formal sums 
\begin{equation} \label{eqn:Novikov}
\Lambda_{\bR} = \left\{ \sum_{i\in \bZ} a_i q^{t_i} \ \big|  \  a_i\in\bC, a_i = 0 \ \textrm{for} \  i\ll 0, t_i\in\bR, t_i \to \infty \right\}
\end{equation}
The Fukaya category has higher-order $A_{\infty}$-operations defined by counts of certain pseudoholomorphic polygons:  these counts assemble into power series which are not known to have positive radius of convergence, so the category is only well-defined over  a field $\Lambda_{\bR}$ of formal power series.  It is also possible to allow as objects of $\scrF(M)$ Lagrangian submanifolds equipped with flat (typically unitary) line bundles; we will not need this extension, but see Remark \ref{rem:moduliofpoints}.
\end{rem}

The first of the conditions imposed on objects of $\scrF(M)$ -- vanishing of the Maslov class -- makes sense whenever $2c_1(M)=0$, and enables the category to be $\bZ$-graded. The second condition -- existence of (relative) spin structures -- enables one to coherently orient the moduli spaces of pseudoholomorphic polygons entering into the definitions of the $\{\mu_{\scrF}^d\}$ and hence define the category with coefficients in a field not of characteristic 2. (In the absence of spin structures, one should take $a_i \in \bZ/2$ in Equation \ref{eqn:Novikov}.)  The third condition is required for the endomorphisms of an object to actually be well-defined.  Recall that Floer cohomology $HF(L,L')$ for a pair of Lagrangian submanifolds $L, L' \subset M$ is defined (roughly) as follows.  One picks a Hamiltonian flow $(\phi_H^t)$ for which $\phi_H^1(L) \pitchfork L'$ is transverse, takes $CF(L,L') = \oplus_{x \in L\cap L'} \bZ\langle x\rangle$, and defines a differential which counts solutions to a perturbed Cauchy-Riemann equation: $dx_+ = \sum_y \# (\mathcal{M}_{x_+, x_-}/\bR) \langle x_-\rangle$, where $\mathcal{M}_{x_+,x_-}$ is the space of solutions $u: \bR\times [0,1] \rightarrow M$ to the perturbed Cauchy-Riemann equation
\begin{equation} \label{eq:dbar_equation}
\partial_s u + J(\partial_t u - X_H(u)) = 0
\end{equation}
with boundary and asymptotic conditions
\[
u(\{0\}\times \bR) \subset L, \ u(\{1\}\times \bR) \subset L'; \ u(s, \pm t) \rightarrow x_{\pm}.
\]
These holomorphic strips come in moduli spaces which, for suitably generic families of compatible almost complex structures, are manifolds, and one counts the zero-dimensional components.  Morphisms in the Fukaya category are by definition the Floer chain groups; the higher order operations of the $A_{\infty}$-structure comprise a collection of maps $\mu_{\sF}^d$ of degree $2-d$, for $d\geq 1$, with $\mu_{\sF}^1$ being the differential:
\[
\mu_{\sF}^d: CF(L_{d-1},L_d) \otimes \cdots \otimes CF(L_0,L_1) \rightarrow CF(L_0,L_d)[2-d]
\]
These have matrix coefficients which are defined by counting holomorphic discs with $(d+1)$-boundary punctures, whose arcs map to the Lagrangian submanifolds $(L_0,\ldots, L_d)$ in cyclic order.  To be slightly more precise, to each intersection point $x$  one actually associates the group $o_{x}$ of coherent orientations -- freely generated by the two orientations of a one-dimensional vector space, subject to the relation that their sum vanishes.  When the Lagrangian submanifolds are relatively spin, the moduli spaces $\scrM(x_0,\ldots,x_d)$ of pseudoholomorphic polygons carry determinant lines which define orientations relative to these coherent orientation spaces, i.e. which yield isomorphisms
\[
\Lambda^{top} T\scrM(x_0,\ldots,x_d) \ \cong \ o_{x_0} \otimes o_{x_1}^{\vee} \otimes \cdots \otimes o_{x_d}^{\vee},
\]
see for instance \cite[Section 12b]{Seidel:FCPLT}. It follows that the isolated points of moduli spaces carry canonical signs relative to the orientation groups associated to intersection points, and the counts $\mu_{\scrF}^d$ are really signed counts if we work in characteristic zero. 

To achieve transversality for the moduli spaces of holomorphic strips, one replaces a fixed  compatible almost complex structure in Equation \eqref{eq:dbar_equation} by a family $J_{t}$ depending on the second factor of a strip. More generally, we replace $J$ by families of almost complex structures indexed by points of the abstract underlying holomorphic disc; such ``universal consistent choices of perturbation data" \cite{Seidel:FCPLT} are constructed inductively over the moduli spaces of holomorphic discs (associahedra).  The resulting maps $\mu_{\sF}^d$ are not chain maps, and hence do \emph{not} naively descend to Floer homology; they are chain-level operations which satisfy the hierarchy of quadratic associativity equations which define an $A_{\infty}$-structure:
\[
\sum_{m,n}  (-1)^{\maltese_n} \mu_{\sF}^{d-m+1} (a_m, \ldots, a_{n+m+1}, \mu_{\sF}^m (a_{n+m},\ldots, a_{n+1}), a_n, \ldots, a_1) = 0.
\]
Crucially, the individual count comprising a given matrix element amongst a particular collection of Lagrangian submanifolds is not well-defined (independent of perturbations or Hamiltonian isotopy), but the entire $A_{\infty}$-structure is well-defined up to quasi-isomorphism.  

\begin{prop} \label{Prop:SelfFloer}
If $L\subset M$ bounds no holomorphic discs for some compatible almost complex structure, then $HF^*(L,L) \cong H^*(L)$. 
\end{prop}

This goes back to Floer \cite{Floer:lagrangian}.  Fixing a Morse function $f:L \rightarrow \bR$ on $L$ defines a Hamiltonian perturbation by the associated Hamiltonian flow $H_f$ of the vector field $X_f = \iota_{\omega}(df)$.  The generators of the Floer complex $L \pitchfork \phi_{H_f}^1(L)$ correspond bijectively to critical points of $f$, and in the absence of bubbling and by a judicious choice of time-dependent almost complex structure, Floer identified the  complex $CF^*(L,L)$ with the Morse complex of $f$.  However, if there are holomorphic discs with boundary on $L$ or $L'$, one can lose control of the compactness of the moduli spaces of perturbed holomorphic strips in a way which breaks the $A_{\infty}$-equations, and even breaks the first such $\mu_{\sF}^1 \circ \mu_{\sF}^1 = d^2=0$.  Floer homology is said to be ``unobstructed" if one can make choices (in general infinitely many, of a delicate inductive nature) to repair this basic deficiency.  The choices of ``bounding chains" on $L$ ``cancel" the errant disc bubbles. Rather than grapple with the deep material of \cite{FO3}, where their properties are addressed in great generality, we opt for a rather low-brow alternative: Floer cohomology is (trivially) unobstructed if for some compatible almost complex structure $J$ there are no holomorphic spheres passing through $L$, and the moduli spaces of holomorphic discs with boundary on $L$ are actually empty. 

\begin{lem} \label{lem:well-defined}
Let $(M,\omega)$ be a four-dimensional symplectic manifold with $2c_1(M)=0$.  For a generic almost complex structure $J$ compatible with $\omega$ and $L\subset M$ a Lagrangian submanifold of Maslov class zero, there are no $J$-holomorphic discs with boundary on $L$ or spheres passing through $L$.
\end{lem}

\begin{proof}
A $J$-holomorphic sphere or disc $u: \Sigma \rightarrow M$  is \emph{somewhere injective} if there is a point $z\in \Sigma$ for which $du(z) \neq 0$ and $u^{-1}(u(z)) = \{z\}$.  If $\Sigma$ is compact without boundary, any simple map (one which does not factor through a branched cover of $\Sigma$ over another curve) contains a dense set of somewhere injective points \cite{McDuff:Examples}.   For curves with boundary, a theorem of  Kwon-Oh \cite{Kwon-Oh} and Lazzarini \cite{Lazzarini} implies the weaker statement  that if $L$ bounds some $J$-holomorphic disc, then it bounds a simple disc.  In both cases, transversality of the Cauchy-Riemann equation can be achieved at simple curves by choosing a generic almost complex structure $J$ on $M$.  

If $M$ is a $2n$-dimensional symplectic manifold with $2c_1(M)=0$, the Riemann-Roch theorem for curves with boundary gives the dimension of the space of unparametrised discs with boundary on $L$ to be $n+\mu_L-3$.  If the Maslov class $\mu_L=0$ and $n=2$ this is negative, hence for generic $J$ the moduli spaces of simple discs are actually empty.  The result of Kwon-Oh or Lazzarini then implies that $L$ bounds no holomorphic discs at all, for generic $J$.  The same argument shows that when $2c_1(M)=0$ the symplectic $4$-manifold $M$ contains no holomorphic spheres for generic $J$.
\end{proof}

Although the Floer Equation (\ref{eq:dbar_equation}) really involves families of almost complex structures $\{J_t\}$, the bubbles that obstruct $\mu^1_{\scrF}\circ\mu^1_{\scrF}=0$ are honest holomorphic discs (for $J_0$ respectively $J_1$ if the bubble appears on the lower respectively upper edge of the strip).  Lemma \ref{lem:well-defined} accordingly implies that the Fukaya category of a four-dimensional symplectic Calabi-Yau manifold can be defined, and any Lagrangian surface $L\subset M^4$ with vanishing Maslov class defines a non-zero object of this category (Proposition \ref{Prop:SelfFloer} shows that it has non-trivial endomorphisms).  However, it is not \emph{a priori} obvious that Hamiltonian isotopic Lagrangian surfaces define isomorphic objects, because in a one-parameter family of Lagrangian submanifolds one does in general expect to encounter bubbles.  In some situations, bubbles can be excluded for topological reasons.

\begin{lem} \label{lem:Ham-independent}
Let $(M,\omega)$ be a four-dimensional symplectic manifold and $L\subset M$ a Lagrangian surface.  If the map $\pi_2(M,L) \rightarrow H_2(M,L)$ vanishes  then the isomorphism class of $L$ in $\sF(M)$ depends only on the Hamiltonian isotopy class of $L$.
\end{lem}

\begin{proof}
It suffices to show that no Hamiltonian image of $L$ bounds any (non-constant) holomorphic disc.  The symplectic form defines an element $[\omega]\in H^2(M,L;\bR)$ and the area of a holomorphic disc $u:(D,\partial D) \rightarrow (M,L)$ is given by the pairing $\int_D \omega = \langle [\omega],[u(D)]\rangle$ between the symplectic form and the image of $u(D)$ in $H_2(M,L)$. The result follows.
\end{proof}

For instance, for Lagrangian submanifolds whose relative $\pi_2$ vanishes, Floer cohomology is defined unproblematically in any dimension and the Fukaya isomorphism type is unchanged by Hamiltonian isotopy.  It follows that linear tori $T^n \subset T^{2n}$ are always elements of the ``strictly unobstructed" Fukaya category $\sF_{so}(T^{2n})$ whose objects are, by definition, Lagrangian submanifolds which for some compatible $J$ bound no holomorphic discs and are intersected by no $J$-holomorphic spheres.  These are the only objects we require when dealing with Fukaya categories of higher-dimensional tori in the rest of the paper.  

Fukaya categories are cohomologically unital but not strictly unital; on the other hand, the cohomological units have geometrically meaningful chain-level representatives.

\begin{lem} \label{lem:units}
The $A_{\infty}$-category $\sF(M)$ can be equipped with distinguished elements $e\in CF^*(L,L)$ which are cycles whose cohomology classes are the units in $H(\sF(M))$.
\end{lem}

\begin{proof}
Given $L$, we have picked a (time-dependent) Hamiltonian function $H(L): M\rightarrow \bR$ for which $\phi_{H(L)}^1(L) \pitchfork L$.  We can regard the generators of the Floer complex $CF^*(L,L)$ as the time-1 chords from $L$ to itself under the Hamiltonian flow of $H(L)$.  The unit element $e$ is obtained as the count of rigid solutions to the perturbed Floer equation with domain a disc with one boundary puncture and boundary condition the family $\phi_{H(L)}^t(L)$.  It is well-known that this represents the cohomological unit, by a standard application of gluing.
\end{proof}

\begin{rem}\label{rem:units=rigidplanes}
Morally, $e$ counts rigid finite-energy half-planes with boundary on $L$; if, following Joyce \cite{Joyce}, we defined $CF(L,L)$ to be the space of Kuranishi chains on $L$, then this would be precisely true.
\end{rem}

\begin{lem} \label{lem:weak-unit-statement}
The perturbation data for the $A_{\infty}$-structure can be chosen such that
\begin{align*}  \mu_{\sF}^{2}(e, e) & = e  \\
\mu_{\sF}^{d}(e, \ldots, e) & = 0  \textrm{ if }d>2. \end{align*}
\end{lem}

\begin{proof}
As in the discussion after Proposition \ref{Prop:SelfFloer}, we can define $CF^*(L,L)$ by choosing a Hamiltonian perturbation of $L$ arising from a Morse function on $L$.  The Floer complex $CF^*(L,L)$ is then concentrated in degrees $0 \leq * \leq n$, from which the second equation follows immediately since the $\mu_{\sF}^d$-operation has degree $[2-d]$.  Taking the perturbing Morse function to have a unique maximum and minimum on each connected component of $L$ implies that the first equation also holds.
\end{proof}

\begin{rem}
Although there is some formal diffeomorphism making $\scrF(M)$ strictly unital, this will in general not be geometric (i.e. the particular matrix entries defining the strictly unital structures $\mu^k_{\Phi_*\scrF}$ will not actually be counts of holomorphic polygons for any choice of perturbation data). 
The condition of Lemma \ref{lem:weak-unit-statement} is much weaker than strict unitality since we are only constraining the behaviour of higher products where \emph{all} the inputs are the identity.
\end{rem}

Once one has set up the Fukaya $A_{\infty}$-category, one can appeal to the general machinery of Section \ref{Subsec:BackgroundAlgebra} to construct categories of twisted complexes, etc.
There are certain situations in which the algebraic operations defining twisted complexes correspond  precisely to geometric operations amongst Lagrangian submanifolds.

\begin{Example} \label{Ex:Dehntwist}
Suppose $V \subset M$  is a Lagrangian sphere.  Then given any Lagrangian submanifold $L\subset M$, one can form either the geometric Dehn twist $\tau_V(L) \subset M$, or the algebraic twist $\scrT_V(L)$ which is (quasi-represents) the cone over the canonical evaluation $HF^*(V,L) \otimes V \rightarrow L$.  These are actually quasi-isomorphic objects of $Tw(\scrF (M))$, by a  theorem of Seidel \cite{Seidel:FCPLT}. 
\end{Example}

Suppose $L_1$ and $L_2$ are oriented Lagrangian submanifolds of $M$ meeting transversely in a single point $p$. The \emph{Lagrange surgery} is a Lagrangian submanifold smoothly isotopic to the connect sum of the $L_i$ (topologically there are two such local surgeries, only one of which is compatible with the fixed local orientations).   The four-dimensional case goes as follows \cite{Seidel:Twosphere}.  Order the $L_i$  and choose a Darboux chart $u: (B^4,0) \rightarrow (M,p)$ near the intersection point which linearises the ordered pair $(L_1,L_2)$ to the (oriented) Lagrangian planes $(\bR^2\times \{0\}, \{0\} \times \bR^2)$.  If $\Gamma\subset \bR^2\backslash \{0\}$ is a smooth embedded curve which lies in the lower right hand quadrant and co-incides with the positive $x$-axis union the negative $y$-axis outside a sufficiently small ball near the origin, the Lagrange handle $H$ is defined by
\[
\left\{ (y_1 \cos(\theta), y_1 \sin(\theta), y_2\cos(\theta), y_2\sin(\theta)) \in \bR^4 \ | \ (y_1,y_2) \in \Gamma, \theta \in S^1 \right\}.
\]
This is diffeomorphic to $S^1\times \bR$, co-incides with $\bR^2\times \{0\} \cup \{0\}\times \bR^2$ outside a compact set,  and is Lagrangian for the standard symplectic form $\sum_j dx_j\wedge dy_j$ on $\bR^4$.  The Lagrange surgery is obtained by replacing $(L_1 \amalg L_2)\cap u(B^4)$ by $u(H)$. Although the operation depends on choices, by viewing different local surgeries as exact Lagrangian graphs in cotangent bundle tubular neighbourhoods of one another, one sees that the resulting Lagrangian submanifold is uniquely defined up to Hamiltonian isotopy.  Gradings and spin structures on the $L_i$ induce a grading and spin structure on the surgery, which is therefore well-defined as an object of the Fukaya category when the original submanifolds are, see \cite[Chapter 10]{FO3}.

\begin{rem} \label{Rem:LagrangeSurgery}
Suppose $L_1$ and $L_2$ meet transversely in a single point $p$, and choose gradings so this point is placed in degree 1 (viewed as a morphism from $L_1$ to $L_2$). The point $p$ necessarily forms a Floer cocycle $[p] \in HF^1(L_1,L_2)$, and in the category of twisted complexes $Tw(\scrF(M))$ one can take the mapping cone of this morphism as in Equation \ref{eqn:cone}.  It is widely expected that  the result is geometrically represented by the Lagrange surgery of $L_1$ and $L_2$ at $p$.  A detailed analysis of holomorphic triangles relevant to this conjectured equivalence is contained in \cite[Chapter 10]{FO3}, but the general statement remains unproven.   If $V\subset M$ is a Lagrangian sphere and $V \pitchfork L = \{p\}$ is a single point, then $\tau_V(L) \simeq L\# V$ is actually the graded Lagrange surgery of $L$ and $V$, so the result is known in that case.
 In Proposition \ref{Prop:SurgeryIsCone}, we will give a direct proof (using Theorem \ref{thm:hms}) that the cone indeed quasi-represents the surgery in the special case of a pair of linear Lagrangian tori in $T^4$ meeting transversely at one point.
\end{rem}

Although there seems to be no compelling example in the literature, one can also motivate the passage from $Tw(\scrF(M))$ to the split-closure $D^{\pi}\scrF(M)$ geometrically.  Suppose one begins with some distinguished finite collection of Lagrangian submanifolds $\{L_j\}$, generating an $A_{\infty}$-subcategory $\scrA \subset \scrF(M)$,  and consider the category of twisted complexes $Tw(\scrA)$.   Heuristically, such a  twisted complex is obtained by resolving collections of transverse intersection points among the $\{L_j\}$, corresponding to taking  cones on a succession of closed degree one morphisms. Such a sequence of Lagrange surgeries could in principle result in a Lagrangian submanifold which is embedded but disconnected.  Taking the idempotent completion of the subcategory $Tw(\scrA)$ allows one to include the individual components of such a disjoint union as objects ``generated" by the $\{L_j\}$.  Except in the case of Riemann surfaces, we emphasise that this dictionary is currently only a heuristic.

The Hochschild cohomology of the Fukaya category can be related to more familiar geometric invariants.  It seems that Seidel \cite{Seidel:deformations} was the first to introduce a map
\begin{equation} \label{eq:seidel-open-closed}
SC^{*}(M) \to CC^*(\sF(M))
\end{equation}
\begin{figure}
   \centering
 \begin{picture}(0,0)%
\includegraphics{seidel-map.pstex}%
\end{picture}%
\setlength{\unitlength}{3947sp}%
\begingroup\makeatletter\ifx\SetFigFont\undefined%
\gdef\SetFigFont#1#2#3#4#5{%
  \reset@font\fontsize{#1}{#2pt}%
  \fontfamily{#3}\fontseries{#4}\fontshape{#5}%
  \selectfont}%
\fi\endgroup%
\begin{picture}(1830,1887)(2386,-2887)
\put(3301,-1111){\makebox(0,0)[lb]{\smash{{\SetFigFont{9}{10.8}{\rmdefault}{\mddefault}{\updefault}{\color[rgb]{0,0,0}$a_0$}%
}}}}
\put(3301,-1936){\makebox(0,0)[lb]{\smash{{\SetFigFont{9}{10.8}{\rmdefault}{\mddefault}{\updefault}{\color[rgb]{0,0,0}$R$}%
}}}}
\put(2401,-1936){\makebox(0,0)[lb]{\smash{{\SetFigFont{9}{10.8}{\rmdefault}{\mddefault}{\updefault}{\color[rgb]{0,0,0}$a_3$}%
}}}}
\put(3301,-2836){\makebox(0,0)[lb]{\smash{{\SetFigFont{9}{10.8}{\rmdefault}{\mddefault}{\updefault}{\color[rgb]{0,0,0}$a_2$}%
}}}}
\put(4201,-1936){\makebox(0,0)[lb]{\smash{{\SetFigFont{9}{10.8}{\rmdefault}{\mddefault}{\updefault}{\color[rgb]{0,0,0}$a_1$}%
}}}}
\end{picture}%
   \caption{}
   \label{fig:seidel-map}
\end{figure}
from the Hamiltonian Floer chain complex $SC^*(M)$ of $M$ to the chain complex which computes Hochschild cohomology.  Given a smooth family $H_{t}$ of Hamiltonian functions on $M$ parametrised by $t \in \bR/\bZ$, and a time-$1$ Hamiltonian orbit $R$ for this time-dependent Hamiltonian, the image of $R$ under the above map is the homomorphism which counts  pseudo-holomorphic discs with an interior puncture converging to $R$, and an arbitrary number of boundary marked points along Lagrangians in $M$.  For example, the $d=3$ component of the natural transformation associated to such an orbit $R$ counts the configurations shown in Figure \ref{fig:seidel-map}; the intersection points $(a_3, a_2 , a_1)$ are the inputs, while $a_0$ is the output.  In the absence of multiply covered discs and spheres, it is not difficult to use the usual count of boundary components of $1$-dimensional moduli spaces to conclude that this map is a chain map.  Using the PSS isomorphism \cite{PSS}, we obtain the desired map
\begin{equation}\label{eq:open-closed-cohomology}
QH^{*}(M) \to HH^*(\sF(M)).
\end{equation}
This map can be defined at the chain level by counting perturbed holomorphic discs with one interior marked point passing through a cycle in $X$, and boundary marked points along Lagrangians in $X$.  This point of view is particularly useful for proving the next result.  Note that Lemma \ref{lem:HHrestriction} gives a canonical map $HH^*(\sF(M)) \rightarrow HF^*(L,L)$ for any object $L\subset M$.

\begin{lem} \label{lem:HHrestrict}
 If $L$ bounds no non-constant holomorphic discs, then the composition
\begin{equation}
 QH^{*}(M) \to HH^*(\sF(M)) \to HF^{*}(L;L)
\end{equation}
agrees with the restriction map on ordinary cohomology.
\end{lem}
\begin{proof}
The composite map counts perturbed holomorphic discs with a single Lagrangian boundary condition, no boundary punctures, and a cycle constraint at an interior marked point.  Since we have imposed a compact boundary condition, the resulting map is chain homotopic to that obtained by deforming the Hamiltonian perturbation in the Floer equation to zero.  At this point the map counts actual holomorphic discs.
Since the only holomorphic discs with boundary on $L$ are constant, the count of discs with one interior marked point labelled by a cycle $R$ in $M$ and one boundary marked point labelled by a cycle $a_0$ in $L$ gives the intersection pairing between the restriction of $R$ to $L$ with $a_0$.
\end{proof}

\section{DG-structures on the derived category of sheaves\label{Subsec:dgBackground}}

The Fukaya category $\sF(M)$ of a symplectic manifold $M$, when well-defined, is an $A_{\infty}$-category, whereas the derived category of coherent sheaves $D^b(X)$ on a projective variety $X$ is only triangulated.  For mirror symmetry, it is important to work with a $dg$-enhancement of $D^b(X)$; there are various essentially equivalent possibilities, using \u{C}ech covers, complexes of injectives, complexes of locally free sheaves or Dolbeault resolutions. The last of these is probably most familiar to differential geometers: if $E\rightarrow X$ and $F \rightarrow X$ are locally free sheaves, then morphisms between these sheaves in the derived category are given by the global Ext-group $\Ext^*(E,F) = H^*(X; Hom(E,F))$.  This is the cohomology group of $X$ with local co-efficients in the bundle of homomorphisms from $E$ to $F$. To obtain a $dg$-category one instead takes morphisms to be the underlying Dolbeault complex $\Omega^{0,*}(X, Hom(E,F);\overline{\partial})$ of homomorphism-valued differential forms.  Working at the level of chain complexes keeps more information in the picture, for instance enabling one to keep track of Massey products.

Dolbeault complexes are not well-suited to working over more general fields $k$, and the mirror of a symplectic Calabi-Yau manifold is typically an algebraic variety over the Novikov field and not over $\bC$.  For definiteness, following \cite{Hovey, LB}, we will instead replace the triangulated category $D^b(X)$ of bounded-below complexes of coherent sheaves by the dg-category whose objects are bounded-below complexes $\{\mathcal{I}^{\bullet}\}$ of \emph{injective} sheaves with bounded coherent cohomology, and whose morphisms of degree $k$ are morphisms of complexes $\mathcal{I}^{\bullet} \rightarrow \mathcal{J}^{\bullet + k}$ (these form a cochain complex in the obvious way).  Denote this enhanced category by $D^b_{\infty}(X)$.  Its underlying cohomological category is equivalent to $D^b(X)$, essentially since coherent sheaves on a projective variety always have injective resolutions; for a proof see \cite{Hovey}.  As with other categories of complexes of sheaves, $D^b_{\infty}(X)$ is already triangulated -- it contains mapping cones -- and split-closed. 

\begin{rem} \label{rem:HPL}
By homological perturbation, Lemma \ref{lem:HPL}, any $A_{\infty}$-category is equivalent to an $A_{\infty}$-category with $\mu^1_{\scrA} = 0$, i.e. to one for which the morphism groups are actually given by the cohomology groups of the original $A_{\infty}$-structure.  We can therefore regard the $dg$-enhancement $D^b_{\infty}(X)$  as equipping the usual derived category $D^b(X)$ with higher-order operations, which will be non-trivial even though the category $D^b_{\infty}(X)$ is a $dg$-category.  Explicitly, one builds such an equivalence by splitting the morphism groups in $D^b_{\infty}(X)$ into summands isomorphic to their cohomology, with zero differential, plus acyclic complements.   Choices of nullhomotopy for these acyclic complements define the higher degree components of the sought-after $A_{\infty}$-equivalence.  Practically, this means that to establish $A_{\infty}$-equivalences, it will typically be sufficient to directly compare $A_{\infty}$-structures on  cohomological (sub)categories, rather than working directly with complexes of injectives.  
\end{rem}

 In contrast to the situation for Fukaya categories, there is a fairly general split-generation criterion due to Orlov \cite{Orlov}.

\begin{thm} \label{thm:orlov}
Let $X$ be an algebraic variety over an algebraically closed field $k=\bar{k}$.  If $L\rightarrow X$ is a very ample line bundle, the successive powers $(L^{\otimes 0}, L^{\otimes 1}, \ldots, L^{\otimes n})$ split-generate the derived category.
\end{thm}

\begin{rem} \label{rem:split-generate}
Bearing in mind Remark \ref{rem:HPL}, Orlov's theorem implies that any fully faithful subcategory $\scrA \subset D^b_{\infty}(X)$ which contains objects quasi-representing the successive powers of an ample line bundle (i.e. for which $H(\scrA)$ contains those powers) has the property that its triangulated split-closed envelope is all of $D^b_{\infty}(X)$.  \end{rem}

Another important point is that passing to this $dg$-enrichment of the derived category of sheaves yields the ``correct" Hochschild cohomology, by a theorem of Lowen and van den Bergh, cf. \cite{LB}.  

\begin{thm} \label{thm:vdb}
Let $X$ be a separated quasi-projective scheme over a field $k$ of characteristic zero. There is a natural isomorphism 
\[
HH^*(D^b_{\infty}(X)) \cong \Ext^*_{X\times X}(\cO_{\Delta},\cO_{\Delta}).
\]
\end{thm}

\begin{rem}
The definition of Hochschild cohomology of $dg$-categories given in \cite{LB}, as the homology groups of an explicit chain complex,  co-incides with that given above for $A_{\infty}$-categories on setting all the higher $\{ \mu^j \}_{j\geq 3}$ to be zero.
\end{rem}

The local-to-global spectral sequence, and the identification of the local ext-group $\cE xt^q_{\cO_{X\times X}}(\cO_{\Delta}, \cO_{\Delta}) \cong \Lambda^q \cT_X$, shows the $\Ext$-group in Theorem \ref{thm:vdb} is computed by a Hodge-type spectral sequence $E_2^{pq} = H^p(X, \Lambda^q \cT_X) \Rightarrow \Ext^{p+q}_{X\times X} (\cO_{\Delta}, \cO_{\Delta})$.  By theorems of Swan \cite{Swan} and Gerstenhaber-Shack \cite{GS}, the spectral sequence degenerates for smooth varieties in characteristic zero.  A particular consequence that we shall need later is:

\begin{cor} \label{cor:rankHH}
If $A$ is an abelian variety of dimension $d$ over $\Lambda_{\bR}$, then $HH^*(D^b_{\infty}(A))$ is isomorphic as a $\Lambda_{\bR}$-algebra to
\begin{equation}
 \bigoplus_{i=0}^{2d} \bigwedge^{i} (\Lambda_{\bR}^{\oplus 2d}).
\end{equation}
\end{cor}
\begin{proof}
 Since $A$ is an abelian variety, the group structure on $A$ trivialises the tangent sheaf $\cT_A$. The graded ring $\oplus_j H^*(A,\Lambda^j\cT_A)$ of polyvector fields, with its usual cup-product structure, is generated in degree one by $H^1(\cO_A) \oplus H^0(\cT_A)$.  The computation of these coherent cohomology groups  in characteristic zero is then well-known \cite{Mumford}, giving respectively the tangent spaces at the origin to the dual and the original variety.  The result follows.

\end{proof}

Besides yielding an object more suitable for comparison with the Fukaya category of a symplectic manifold, enhancing the derived category to a $dg$-category strictly improves its behaviour in various fundamental ways. For us, the most important of these will allow us
to pass between categories of functors and categories of sheaves on a product variety, using the following beautiful result of To\"en \cite[Theorem 8.9 \& 8.15]{Toen}.

\begin{thm} \label{thm:toen}
There is a quasi-equivalence of $dg$-categories
\begin{equation} \label{eqn:toen}
D^b_{\infty}(X\times X) \ \simeq \ \End(D^b_{\infty}(X)).
\end{equation}
\end{thm}

Note that on a smooth scheme all complexes of coherent sheaves are perfect complexes;  To\"en's $dg$-enhancement $L_{parf}(X)$ has as objects the complexes of fibrant-cofibrant objects, which for the injective Quillen model category structure on sheaves exactly co-incides with complexes of injectives, hence $L_{parf}(X) = D^b_{\infty}(X)$.  For arbitrary schemes, To\"en's notion (his $\bR\underline{Hom}_c$) of the appropriate category of $dg$-functors is not  obviously that coming from our description of $A_{\infty}$-functors and the bar resolution (he imposes a continuity requirement on the functors, namely that the induced functors on the homotopy categories of complexes of quasi-coherent sheaves commute with direct sums).  However, for smooth projective schemes this subtlety is removed in \cite[Theorem 8.15]{Toen}.

\begin{rem}
The analogue of Theorem \ref{thm:toen} for the derived category $D^b(X)$ is certainly false.  Although no precise analogue of the theorem is known for Fukaya categories, one can sometimes derive a reasonable substitute for ``generalized" Fukaya categories by using the quilt theory of Mau-Wehrheim-Woodward. We review this theory next.
\end{rem}

\section{Quilted Floer theory\label{Subsec:BackgroundQuilts}}
The proof of the main ``generation" statement Theorem \ref{thm:generation}  relies on the $A_{\infty}$-analogue of Wehrheim-Woodward's Lagrangian correspondences through quilts.  Briefly, Wehrheim and Woodward \cite{WW} introduce a category $\sF^{\#}(M)$
consisting of generalized Lagrangian correspondences, and together with Mau \cite{MWW,Mau} associate to a Lagrangian correspondence $L^{\flat} \subset M^-\times N$ an $A_{\infty}$-functor $\cF_{L^{\flat}}:\sF^{\#}(M) \rightarrow \sF^{\#}(N)$.  Our essential requirement is that this association $L^{\flat} \mapsto \cF_{L^{\flat}}$  can be extended to define an $A_{\infty}$-functor
\begin{equation}
\Phi \co \sF(M^{-} \times M) \to \End(\sF^{\#}(M)) \end{equation}
where the right hand side is the category of $A_{\infty}$-endofunctors of $\sF^{\#}(M)$. The rest of this subsection reviews what we expect and need of their theory at slightly greater length.  This treatment is necessarily somewhat cursory, and we refer to their original papers and preprints for fuller discussion.  

\begin{rem}
The original paper of Wehrheim and Woodward \cite{WW} on the cohomological functor $H(\Phi)$ works only with Lagrangian submanifolds of sufficiently positive Maslov class (whereas we require Lagrangians of Maslov class zero).  However, this assumption enters their work to prevent bubbling of holomorphic discs and of ``figure-eights.''  We avoid the former simply because there are no such discs with boundary on a Lagrangian surface of Maslov class zero as long as the complex structure is chosen generically, while the latter arise only in proving that the functor associated to a geometric composition of Lagrangian correspondences co-incides with the composition of the functors associated to the individual correspondences.  We make no use of such a statement. 
\end{rem}

A generalized Lagrangian submanifold of a symplectic manifold $M$ comprises an integer $k\geq 1$ and a sequence of symplectic manifolds and Lagrangian correspondences
\[
(\{pt\}=M_0, M_1,\ldots, M_k=M); \qquad L_{i,i+1} \subset M_i^-\times M_{i+1}, \ 0\leq i\leq k-1
\]
where $(X,\omega)^-$ is shorthand for $(X,-\omega)$.  These form the objects of a category $H(\sF^{\#}(M))$. The morphisms between two such are Floer homology groups defined as follows:  take objects
\[
\mathbb{L} = pt \stackrel{L_1}{\longrightarrow} M_1  \stackrel{L_{12}}{\longrightarrow} M_2 \cdots  \stackrel{L_{k-1,k}}{\longrightarrow}M_k = M
\]
and (a chain of perhaps different length)
\[
\mathbb{L}' = pt \stackrel{L'_1}{\longrightarrow} M'_1  \stackrel{L'_{12}}{\longrightarrow} M'_2 \cdots  \stackrel{L'_{r-1,r}}{\longrightarrow}M'_r = M.
\]
We obtain two Lagrangian submanifolds $\mathbb{L}_{\pm}$ of 
\[
\{pt\}\times M_1 \times M_2 \times \cdots \times M_k \times M'_r \times M'_{r-1} \times \cdots \times M'_1 \times \{pt\}.
\]
There are various cases depending on the parity of $k$ and $r$, but supposing for instance that $k$ is even and $r$ odd, these are
\[
\mathbb{L}_+ \ = \ L_1 \times L_{23} \times \cdots \times L_{k-2,k-1} \times \Delta \times L'_{r-2,r-1} \times \cdots \times L'_{21}
\]
and 
\[
\mathbb{L}_- \ = \ L_{12} \times L_{34} \times \cdots \times L_{k-1,k} \times L'_{r-1,r} \times \cdots \times L'_{32}\times L'_1.
\]
Here $\Delta \subset M_k\times M'_{r} = M\times M$ denotes the diagonal, and the signs of the symplectic forms alternate along the chain; if the parities are different, the $\Delta$-term may occur in the other factor or not at all. 
\begin{rem} \label{rem:transposeLags}
Note that we have transposed both the second sequence of symplectic manifolds \emph{and} the Lagrangian correspondences $L_{i,i+1}'$. Thus, for instance, the graph of a symplectomorphism $L_{12}' \subset (M_1')^- \times M_2'$ goes to the graph of the inverse inside $(M_2')^- \times M_1'$, and a product $L_{12}' = K_a' \times K_b'$ would be replaced by $K_b' \times K_a'$.  In \cite{WW} this transposition operation is denoted $L \mapsto L^t$, but we suppress the notation to avoid (more) clutter.
\end{rem}
In any case, we can then define:
\begin{equation}\label{eqn:morphism}
\textrm{Mor}_{H\sF^{\#}(M)}(\mathbb{L},\mathbb{L}') = HF(\mathbb{L}_+, \mathbb{L}_-).
\end{equation}
(Relative) spin structures and gradings on the individual $L_{j,j+1}$ and $L'_{i,i+1}$ induce corresponding structures on the $\mathbb{L}_{\pm}$; if all the constituent correspondences are monotone with the same monotonicity constant or are known to bound no holomorphic disc for some almost complex structure, the same is true of $\mathbb{L}_{\pm}$ for suitable product almost complex structures; in short, under the usual assumptions, the Floer group on the right side of Equation \ref{eqn:morphism} is both well-defined and $\bZ$-graded.  Note that typically one will require Hamiltonian perturbations to achieve transversality of the Lagrangians;  split Hamiltonians of the shape $H=\sum_{i,j} H_j + H_i'$, $H_j: M_j \rightarrow \bR$, $H_i':M_i' \rightarrow \bR$, are actually sufficient \cite{WW}, in which case the generators of the underlying Floer chain group $CF(\mathbb{L}_+, \mathbb{L}_-)$ will comprise tuples of Hamiltonian chords $x_j: [0,1]\rightarrow M_j$ with $(x_j(1),x_{j+1}(0)) \in L_{j,j+1}$, and similarly for the $'$-ed manifolds and Lagrangians. 

Although the morphism groups themselves are classical Floer groups, the multiplication which defines composition in the category $H(\sF^{\#})$ -- and the higher order operations of the underlying chain-level $A_{\infty}$-structure -- is formulated in terms of maps of ``quilted" Riemann surfaces.  Informally, a quilted Riemann surface is a Riemann surface with boundary punctures and strip-like ends and with a number of \emph{seams} (embedded real analytic arcs) running through the interior of the surface between boundary punctures. We think of the surface as a collection of distinct domains with the seams as partial boundary identifications between these domains.  The case of most importance here is the following.  Say a \emph{quilted $(d+1)$-marked disc} is a disc $D\subset \bC$ with boundary marked points $\{z_0, z_1, \ldots, z_d\}$ and with a distinguished horocycle -- the seam -- at the point $z_0$ (i.e. an interior circle tangent to $\partial D$ at $z_0$; note this is a real analytic submanifold).

The set of discs with $d+1$ marked point forms a moduli space which may be compactified using stable discs to give a model for Stasheff's associahedron. This is a cell complex whose codimension one boundary facets enumerate the terms of the $A_{\infty}$-equations:  this is why counts of holomorphic discs endow the Fukaya category with the structure of an $A_{\infty}$-category.  Mau and Woodward introduced a notion of nodal quilted discs, and proved:

\begin{prop}\cite{MW}
The moduli space of nodal stable quilted $(d+1)$-marked discs is a convex polytope homeomorphic to the multiplihedron.
\end{prop}

The multiplihedron, also introduced by Stasheff, has vertices which index ways of not only bracketing a string of variables (as in the associahedron) but also applying an operation: for $3$ marked points one gets six possibilities: $\cF(x_1(x_2x_3))$, $\cF((x_1x_2)x_3)$, $\cF(x_1)\cF(x_2x_3)$, $\cF(x_1x_2)\cF(x_3)$, $(\cF(x_1)\cF(x_2))\cF(x_3)$ and $\cF(x_1)(\cF(x_2)\cF(x_3))$, which form the vertices of a hexagon.  The codimension one boundary faces of the multiplihedron correspond to the terms of the quadratic $A_{\infty}$-functor equation
\begin{multline*}
\sum_{i_j, k} \mu_{\sB}^k(\cF(a_n,\ldots,a_{i_1}),\cF(a_{i_1-1},\ldots, a_{i_2}), \ldots, \cF(a_{i_k-1},\ldots a_1)) \\
= \sum_d (-1)^{\maltese_{j-d}} \cF(a_n,\ldots, a_{j}, \mu_{\sA}^d (a_{j-d},\ldots, a_{j-d-1}), a_{j-d},\ldots a_1) 
\end{multline*}
for an $A_{\infty}$-functor $\cF: \sA \rightarrow \sB$. This is a consequence of the fact that the boundary strata of a multiplihedron are products of lower-dimensional associahedra and multiplihedra.  

Any quilted Riemann surface whose boundary components are labelled by Lagrangian submanifolds and whose seams are labelled by Lagrangian correspondences, in both cases compatibly with an assignment of intersection points (chords) to strip-like ends, determines an elliptic boundary value problem. This problem studies a collection of holomorphic maps, one defined on each subdomain of the surface (closure of a connected component of the complement of the seams), subject to Lagrangian boundary conditions as prescribed by the labelling data along boundaries and seams. Near any given strip-like end, the collection of maps can \emph{locally} be regarded as a map into a product of symplectic manifolds with some fixed Lagrangian boundary condition; analogously, one can repackage the morphism groups $HF(\bL_+,\bL_-)$ as counts of quilted cylinders into the original manifold $M$ with the $L_{i,i+1},$ and $L'_{j,j+1}$ as seam conditions for a collection of seams running down the cylinder, cf. \cite[Figure 7]{WW}.  Locality of the usual elliptic estimates then gives exponential decay of finite energy solutions; one can achieve transversality for the Cauchy-Riemann equations by generic choices of almost complex structure on the various factors $M_j, M_i'$, \cite[Remark 3.4.2]{WW}.   

Similarly, one can consider moduli spaces of holomorphic quilts, allowing the holomorphic structure on the domain to vary in a finite-dimensional family.  The relevant ``universal" families of surfaces over the multiplihedra are constructed in \cite{Mau}, where she also introduces \emph{consistent choices of perturbation data} inductively over these universal families, following the strategy adopted in \cite{Seidel:FCPLT}. These take the shape of families of one-forms on the surfaces of the universal fibration with values in split Hamiltonian functions on the product $\prod M_j \times \prod M_i'$.  The perturbation data can be chosen independently for different $\{M_j, M_i'\}$ and $\bL$, provided for any fixed collection of such branes it is consistent with respect to the inductive construction of the universal families of surfaces themselves. 

The $A_{\infty}$-structure on $CF(\bL_+,\bL_-)$ is derived from certain moduli spaces of quilted discs which comprise a $(d+1)$-marked disc with a collection of parallel strips attached to the boundary arcs, cf. Figure \ref{fig:quilt_pair_pants} and \cite[Figure 20]{WW}.  These additional strips have fixed width, and the moduli spaces of these configurations are just the same as of the unquilted component, hence form an associahedron.   The resulting $A_{\infty}$-structure is cohomologically unital, with the cohomological units  $e_{\bL}\in CF(\bL, \bL)$ counting rigid quilted discs with a single out-going end, \cite[Figure 22]{WW}. If $\bL = (L_1, L_{1,2}, \ldots, L_{k,k+1})$ such a count therefore comprises a half-plane mapping to $M$, with boundary condition in $L_{k,k+1}$, and a collection of boundary parallel strips attached along parallel seams, with successive boundary conditions $L_{i,i+1}$ for $1\leq i\leq k$; the outermost strip maps to $M_1$ with boundary condition the Lagrangian submanifold $L_1\subset M_1$.  There is a \emph{gluing theorem}, due to Mau \cite{Mau}, which reconstructs smooth quilted discs from nodal ones: roughly, given any  collection $\cC$ of nodal quilted discs, it asserts the existence of a map $\delta: (R,\infty) \rightarrow (\mathcal{M}_{d,1})^{1-dim}$ into a 1-dimensional component of a space of smooth quilted discs, for some $R \gg 0$, for which $\delta(t) \rightarrow \cC$ converges in the Gromov topology to the given nodal configuration as $t\rightarrow \infty$ (this in particular implies $C^0$-convergence).   For spaces of quilted discs indexed by the associahedron, this result allows one-dimensional moduli spaces to be compactified in a way which makes the associated counting operations $\mu^d_{\sF^{\#}}$ satisfy the usual quadratic $A_{\infty}$-relations; compare \cite[Proposition 12.3]{Seidel:FCPLT}, and which makes  $e_{\bL}$ into a cohomological unit.   More generally, one can consider the moduli spaces $\mathcal{M}_{d,1}$ of quilted discs parametrised by the multiplihedron, and the gluing theorem  controls the geometry of the boundaries of one-dimensional components of these moduli spaces, which means that the counts of quilted discs indeed reflect the combinatorial boundary structure of these spaces and hence the $A_{\infty}$-functor equations.    In the presence of spin structures, the moduli spaces of holomorphic quilted discs are again orientable relative to coherent orientations on the tuples of chords which comprise intersection points of generalised Lagrangians.

These results go a long way, but do not yet take into account all of the available structure.
The required output of the theory of Mau-Wehrheim-Woodward \cite{MWW} amounts to the following \emph{Axiom} (the first part of which has been established).

\begin{axi} \label{claim:mww}
Assume all Lagrangians below are spin, of Maslov class zero, and admit a choice of almost complex structure for which they bound no non-constant stable holomorphic map from a Riemann surface of genus zero with one boundary component.
\begin{enumerate}
\item  To every Lagrangian correspondence $L^{\flat} \subset M^-\times N$, there is an $A_{\infty}$-functor $\cF_{L^{\flat}}: \sF^{\#}(M) \rightarrow \sF^{\#}(N)$ defined on objects by
\begin{align*}
\bL & \rightarrow  \cF_{L^{\flat}}(\bL) \\
(L_1, L_{23}, \ldots, L_{k-1,k}) & \mapsto  (L_1,L_{23},\ldots, L_{k-1,k},L^{\flat})
\end{align*}
and on morphisms and higher products by a signed count of quilted discs:
\[
CF(\bL_{d-1},\bL_d) \otimes \cdots \otimes CF(\bL_0,\bL_1) \rightarrow CF(\cF_{L^{\flat}}(\bL_0), \cF_{L^{\flat}}(\bL_d))[1-d]
\]
taking chords $(\underline{p}_d,\ldots, \underline{p}_1) \mapsto \sum_{\underline{q}} (-1)^{\delta}\#\mathcal{M}_{d,1}(\underline{p}_d,\ldots,\underline{p}_1,\underline{q}) \langle \underline{q} \rangle$ with the sign $(-1)^{\delta}$ dictated by the appropriate determinant line bundle. 

\item The gluing maps of quilted discs are compatible with the coherent orientations on moduli spaces of quilted discs developed in \cite{WW:orientquilts}.  The categories $\sF^{\#}(M)$ and functors $\cF_{L^{\flat}}$ are defined over the characteristic zero Novikov field $\Lambda_{\bR}$ and are absolutely $\bZ$-graded.

\item The assignment $L^{\flat} \mapsto \cF_{L^{\flat}}$ defines an $A_{\infty}$-functor $\Phi: \sF(M^-\times M) \rightarrow \End(\sF^{\#}(M))$.  If we assume $M$ is itself \emph{spin}, this takes the diagonal $L^{\flat} = \Delta \subset M^-\times M$ to the identity $\id_{\sF^{\#}(M)}$.
\end{enumerate}
\end{axi}

The restriction to spin $M$ in the last part of the \emph{Axiom} avoids the need to talk about Pin structures on the diagonal in the more general case. 

\begin{rem}
All the parts of the \emph{cohomological} version of this statement have been established, cf. \cite{WW} and references therein.  The crucial final part of the \emph{Axiom} is the subject of ongoing work-in-progress by Mau, Wehrheim and Woodward. Moreover, for the proof of the generative criterion Theorem \ref{thm:generation}, we shall only need generalised Lagrangian correspondences in which every constituent is either a product of Lagrangians coming from the respective factors, or a diagonal.  In these cases the quilt theory should simplify considerably.
\end{rem}

\begin{rem}
One can use perturbations of the families of almost complex structures parametrised by the points of a quilt which are compactly supported in the interior of that quilt to achieve regularity for moduli spaces.  That is why the theory behaves well for a given Lagrangian submanifold as soon as there is a single almost complex structure for which it bounds no holomorphic disc, cf. the preamble to the \emph{Axiom}.  Note that, a priori,  the choice of such an almost complex structure becomes an additional datum on the objects (branes) of the strictly unobstructed Fukaya category $\sF_{so}(M)$.  In general, this is to be remedied by techniques developed in \cite{FO3}.
\end{rem}

\begin{rem}
The technically hardest results of \cite{MWW} concern the quasi-equivalence of composition of functors and the functor for a composition of Lagrangian correspondences, of which we make no use (either at the $A_{\infty}$ or the homological level; in our framework the relevant theorem would assert that two \emph{distinct} functors were quasi-isomorphic).
\end{rem}

\section{Homological mirror symmetry for the 2-torus}

\subsection{Split-generation}
There are few split-generation results for Fukaya categories, but one follows from Example \ref{Ex:Dehntwist}. This is analogous to \cite[Lemma 6.2]{Seidel:HMSgenus2} which considers surfaces of genus $>1$.  

\begin{prop} \label{prop:seidel}
 Let $M \cong T^2$ be a closed two-dimensional symplectic torus. A meridian and a longitude split-generate $\scrF(M)$.
  \end{prop}

\begin{proof}
Let $L$ be any homotopically essential closed curve representing an object of $\scrF(M)$.  Let $V_1$ respectively $V_2$ denote the chosen meridian and longitude; there is a positive relation in the mapping class group $(\tau_{V_1}\tau_{V_2})^{6k} = 1 \in SL_2(\bZ)$ for $k\geq 1$.  
The interpretation of the Dehn twists as twist functors, cf. Example \ref{Ex:Dehntwist}, gives a canonical morphism $L \rightarrow (\tau_{V_1} \tau_{V_2})^{6k} (L) \cong L$ which defines an element $\mathcal{C}$ of $HF^0(L,L) \cong H^0(L)$.  Again by Example \ref{Ex:Dehntwist}, the iterated twist $(\tau_{V_1}\tau_{V_2})^{6k}(L)$ and hence the cone of the morphism $L \rightarrow (\tau_{V_1}  \tau_{V_2})^{6k} (L)$ necessarily lie in the triangulated closure of $V_1$ and $V_2$.  We shall show that when $k=2$ the class $\mathcal{C}$ vanishes, hence the cone is isomorphic to a direct sum of two copies of $L$.  This will imply $L$ is split-generated by $V_1$ and $V_2$, as required.

Let $\pi: W\rightarrow \bP^1$  denote  the closed symplectic 4-manifold which is the total space of the  Lefschetz fibration  with fibre $M$ defined by the positive relation $(\tau_{V_1} \tau_{V_2})^{12} = 1$.  Equip $W$ with an almost complex structure $J$ taming the symplectic form and making the projection map $\pi$ pseudo-holomorphic;  it is well-known that transversality for holomorphic curves not containing components lying in fibres of $W$ can be achieved within this class, see for instance \cite{DonaldsonSmith}.  When the spaces of holomorphic sections of $\pi$ (in all possible homology classes $[\beta]$) are compact and of the expected dimension, they define -- by evaluation and Poincar\'e duality -- an even-dimensional cohomology class $\mathcal{C}(\pi) \in H^{ev}(M;\Lambda_{\bR})$.  
From the proof of the long exact sequence in Floer cohomology \cite{Seidel:LES},  one sees that the class $\mathcal{C} \in H^0(L)$ is geometrically realised by the restriction $H^0(M)\rightarrow H^0(L)$ of the zero-degree component of $\mathcal{C}(\pi)$.

Given $[\beta] \in H_2(W)$, the space of holomorphic sections of $\pi$ in the class $[\beta]$ has complex virtual dimension $\langle c_1(T^{vt}W), [\beta]\rangle + 1$, where $T^{vt}$ denotes the vertical tangent bundle. In the case at hand,  $W$ is well-known to be a $K3$ surface, in particular $c_1(W)=0$, which implies that the virtual dimensions of such spaces of sections are all negative, hence the relevant moduli spaces are empty.  This  forces $\mathcal{C}=0$, completing the proof.  
\end{proof}

\begin{rem}
If $g(M) > 1$ and there is a positive relation $\prod_i \tau_{V_i} = \id$ amongst the Dehn twists in circles $V_i \subset M$, one can again consider the associated Lefschetz fibration $\pi: W\rightarrow \bP^1$, but now bubbling is excluded for topological reasons (fibres of $\pi$ contain no spherical components).  The only components of $\mathcal{C}(\pi)$ in $H^0(M)$ arise from sections of square zero, by adjunction, but these can be excluded by positivity of intersections for holomorphic curves in four-manifolds \cite{McDuff:Positivity}; thus Seidel deduces the $\{V_j\}$ split-generate $\scrF(M)$.
\end{rem}

\subsection{The theorem}
Recall the Novikov field $\Lambda_{\bR}$ comprises formal sums 
\begin{equation}
\Lambda_{\bR} = \left\{ \sum_{i\in \bZ} a_i q^{t_i} \ \big|  \  a_i\in\bC, a_i = 0 \ \textrm{for} \  i\ll 0, t_i\in\bR, t_i \to \infty \right\}.
\end{equation}
Define $\sigma: \Lambda_{\bR} \rightarrow \bR\cup\{\infty\}$ by $\sigma (\sum a_i q^{t_i}) = t_{min}$ (the smallest power of $q$ which occurs in the expression with non-zero coefficient) and $\sigma(0)=\infty$. Then $|\cdot| = e^{-\sigma}: \Lambda_{\bR} \rightarrow \bR_{\geq 0}$ defines a valuation on $\Lambda_{\bR}$, with respect to which it becomes a complete non-Archimedean field (note that $|q|<1$).  Background on algebraic geometry over such fields can be found in \cite{FresnelPut}, but in this paper nothing sophisticated is required: we only really use the fact that $\Lambda_{\bR}$ is algebraically closed of characteristic zero.  We recall from the Introduction:

\begin{defin} \label{defin:Tate}
The Tate curve $E=E_{\Lambda_{\bR}}$ is the elliptic curve defined algebraically by its ring of functions 
\[ \Lambda_{\bR}[x,y] \big/ \{y^2 + xy = x^3 + a_4(q) x + a_6(q)\}, \] for series $a_4, a_6 \in \bZ[[q]]$ defined by:
\[
s_k(q) = \sum_{m\geq 1} \frac{m^kq^m}{1-q^m}; \ a_4(q) =  -5s_3(q);  \ a_6(q) = (-5s_3(q)-7s_5(q))/12.
\]
We will identify $E$ with its analytification, which is exactly $\Lambda_{\bR}^*/\langle q\rangle$, where $\Lambda_{\bR}^*$ denotes the subspace of non-zero elements.  
\end{defin}

Since the relevant series converge for $|q|<1$, one can informally regard the Tate curve as a family of elliptic curves over a disc, with nodal central fibre (this family has ``maximal unipotent monodromy" in the usual sense).  The appearance of a family reflects the fact that, on the mirror, any symplectic torus has an area which can be scaled, giving rise to a one-parameter family of Fukaya categories; for technical convergence reasons (which can be circumvented in complex dimension one but not thereafter), we in fact regard a symplectic manifold as defining a single category over $\Lambda_{\bR}$. 
In the rest of this section we will explain the proof of the following:

\begin{thm} \label{thm:mirror_symmetry_elliptic}
The split-closed derived Fukaya category of $T^{2}$ equipped with a symplectic form of area $A$ is equivalent to the bounded derived category of coherent sheaves on the elliptic curve $E_A=\Lambda^{*}_{\bR} /  \langle q^{A} \rangle.$
\end{thm}

\begin{rem}
The statement of the theorem asserts an equivalence of triangulated categories.  The proof essentially constructs a quasi-isomorphic embedding $\cA(T^2) \hookrightarrow D^b_{\infty}(E_A)$ of a certain subcategory of the Fukaya category into the $dg$-enhancement of $D^b(E_A)$ introduced in Section \ref{Subsec:dgBackground}.  This induces an equivalence of the split-closed category of twisted complexes $Tw^{\pi}(\cA(T^2)) \simeq D^b_{\infty}(E_A)$.  The triangulated equivalence of the theorem is then obtained by passing to cohomology.  In fact, as indicated in Remark \ref{rem:HPL}, by suitable appeal to homological perturbation we will avoid dealing directly with complexes of injectives.
\end{rem}

\begin{rems}
There are several statements of homological mirror symmetry for elliptic curves in the literature.  Polishchuk and Zaslow \cite{PZ} identify the underlying cohomological categories of the elliptic curve and its mirror.  Polishchuk later refined this by checking equivalence of Massey products \cite{P1}, and verified that the $A_{\infty}$-structure is in a sense  determined by a triple product \cite{P2}.  The latter papers work with an elliptic curve over an arbitrary field $k$ on the algebraic side and specialise to $k=\bC$ on the symplectic side (it follows from direct computation that the relevant power series defining the $A_{\infty}$-operations are all convergent, so the specialisation to $k=\bC$ is legitimate).  The closest result in the literature to Theorem \ref{thm:mirror_symmetry_elliptic} is due to Gross \cite{Gross}, who works over an integral version of our Novikov ring and constructs an embedding of a subcategory of the derived category of sheaves on (a formal scheme underlying) the Tate curve into the Fukaya category of the curve of area 1. 
\end{rems}


\subsection{The proof}

Fix $A\in\bR_{>0}$ and consider the torus of area $A$.  Fix a (Lagrangian!) fibration $\pi:T^2\rightarrow \bR/A \bZ$.  Let $\cA(T^2)$ denote the subcategory of $\sF(T^2)$ generated by the two objects $L_s$ and $L_f$ given by any fixed choice of section and  fibre  of $\pi$, each equipped with the trivial choice of flat local system.  Note that $[L_s]$ and $[L_f]$ form a basis for $H_1(T^2;\bZ)$ or $H_1(T^2;\Lambda_{\bR})$. The subcategory $\cA(T^2)$ split-generates the Fukaya category by Proposition \ref{prop:seidel}.

Let $E_A$ denote the elliptic curve over $\Lambda_{\bR}$ with analytification $E_A = \Lambda_{\bR}^*/\langle q^A\rangle$.  Let $\cA^{\vee}(E_A)$ denote the subcategory of  the  derived category of sheaves on $E_A$ generated by the structure sheaf $\cO$ and the skyscraper sheaf of a closed point $\cO_p$.    As usual, we write $\cO(p)$ for the line bundle which sits in an exact sequence
\begin{equation} \label{eqn:basic-cone} 0 \to \cO \to \cO(p) \to \cO_p \to 0. \end{equation}
Again, on the algebraic side,  the structure sheaf and skyscraper sheaf split-generate the derived category of the elliptic curve. Explicitly, iterating Equation \ref{eqn:basic-cone}, i.e. iteratively tensoring it with $\cO(\pm p)$, shows that $\cO(kp)$ is in the triangulated envelope of $\langle \cO, \cO_p \rangle$ for all $k\in \bZ$; the result then follows from Orlov's theorem \ref{thm:orlov}.  (In this dimension one could also give a direct argument: every complex is quasi-isomorphic to its cohomology sheaves, torsion-free sheaves are locally free, and locally free sheaves have been classified by Atiyah.)


Informally, we will compare $\cA(T^2)$ with $\cA^{\vee}(E_A)$, the relevant correspondence at the level of objects coming from $\cO \leftrightarrow L_s$ and $\cO_p \leftrightarrow L_f$.  Note that all four objects are spherical, and $\Ext^*(\cO,\cO_p) \cong HF^*(L_s,L_f) \cong \Lambda_{\bR}$ (concentrated in degree 1). Rather than writing down an explicit functor between the two categories, we will appeal to an algebraic classification result of Polishchuk. He proved in \cite{P3} that  the $A_{\infty}$-structure on the subcategory $\Gamma \cA^{\vee}$ of $D^{b}_{\infty}(E_A)$ whose objects are exactly the powers of $\cO(p)$ is uniquely determined by the cohomological category $\sC=H(\Gamma \cA^{\vee})$ and by its lack of formality: the category $\sC$ admits a unique non-formal $A_{\infty}$-structure up to $A_{\infty}$-equivalence.  This is proved by a Hochschild cohomology computation:  $HH^2(\sC)$ has rank 1, so the trivial structure has a unique deformation, cf. the discussion after Lemma \ref{lem:HHrestriction}.  On the symplectic side, we therefore introduce the subcategory $\Gamma \cA(T^2)$ of $Tw(\cA(T^2))$ generated by the iterated Dehn twists of $L_s$ by $L_f$.  It is an immediate consequence of \cite{PZ}, by direct computation, that there is an equivalence of cohomological categories
\[
H(\Gamma \cA^{\vee}) \simeq H(\Gamma \cA(T^2)).
\]
Strictly, we are eliding an important issue.  Polishchuk and Zaslow \cite{PZ} work only with transversal collections of Lagrangian submanifolds which is one of the reasons why they only work on the cohomological level.  However, one can use Hamiltonian perturbations to deal with non-transversal, in particular co-incident, Lagrangian submanifolds in the manner of Seidel \cite{Seidel:FCPLT}, thereby obtaining a genuine (cohomologically unital) $A_{\infty}$-category $\sF$.   In real dimension 2, implementing the Hamiltonian perturbations is fairly straightforward.

It suffices therefore to prove the lack of formality of the mirror category $\Gamma \cA(T^2)$.  One proof proceeds by imitating Polishchuk's proof of the lack of formality of the derived category, using an explicit Massey product counting holomorphic quadrilaterals  bound by four circles Hamiltonian isotopic to $L_s$, $\tau^{2}_{L_f} L_s$,  $\tau^{4}_{L_f}L_s$, and $L_s$.  These quadrilaterals are manifest in the universal cover of $T^2$;  the idea is to reproduce the Koszul exact sequence appearing on page 418 of \cite{P2} with a Lagrangian mirror to a line bundle of degree $2$.

Alternatively, and more succinctly, we give an argument based on the rank of certain $K_0$-groups (these were defined in Equation \ref{eqn:K0}).  By construction $H(\Gamma \cA(T^2))$ has objects indexed by an integer corresponding to the power of the Dehn twist about $L_f$, and any two different objects are not isomorphic.  Moreover, for any two objects $L_0$ and $L_1$ which are not isomorphic, the image of the composition
\begin{equation*} HF^*(L_1,L_0) \otimes HF^*(L_0,L_1) \to HF^*(L_0,L_0) \end{equation*} 
factors through the degree $1$ part of $HF^*(L_0,L_0)$.  In particular, to each object $L$ in $H(\Gamma \cA(T^2))$, we may assign a unital $dg$ functor
\begin{equation} \label{eqn:functors} H(\Gamma \cA(T^2))  \to Ch  \end{equation}
taking $L$ to the ground field $k$, and all other objects to $0$ (prescribing this action on objects, together with the requirement of being a unital $dg$ functor, determines the map on morphism spaces in this special case).

If $ \Gamma \cA(T^2) $ were formal, so quasi-isomorphic to the trivial $A_{\infty}$-structure on $H\Gamma \cA(T^2)$, we would conclude that the $K_0$-group of its category of twisted complexes has infinite rank. Indeed,  the functors (\ref{eqn:functors}) induce maps on $K_0$ distinguishing the $K$-theory classes of any two Lagrangians in $\Gamma\scrA(T^2)$.  This contradicts the fact that the triangulated closure of $ \Gamma \cA(T^2) $ agrees with the triangulated closure of the category with objects $L_f$ and $L_s$, and hence must have a $K_0$ group of rank at most $2$.

The upshot is that the subcategories 
\[
\Gamma \cA^{\vee} \subset D^b_{\infty}(E_A) \qquad \textrm{and} \qquad \Gamma \cA(T^2) \subset Tw(\scrF(T^2))
\]
carry equivalent $A_{\infty}$-structures; in fact they are quasi-isomorphic by an $A_{\infty}$-functor acting trivially on cohomology and acting on objects according to our original informal prescription. Since the subcategories $\cA$ and $\cA^{\vee}$ split-generate on their respective sides, and passage to a split-closed triangulated envelope is purely algebraic, this shows that the corresponding split-closed derived categories are equivalent, and passing finally to cohomology yields the statement of Theorem \ref{thm:mirror_symmetry_elliptic}.


\begin{rem} \label{rem:moduliofpoints}
The space $E_A$ can be identified with the moduli space of skyscraper sheaves on $E_A$, all of which are obtained from $\cO_p$ by applying a translation automorphism of the elliptic curve (these come from an algebraic group structure on the underlying algebraic variety, for instance by the ``rigid GAGA" principle \cite{FresnelPut}).  To see a corresponding moduli space of objects on the mirror symplectic torus $T^2$ without appealing to idempotent closure, one must enlarge the Fukaya category and allow objects comprising a Lagrangian submanifold together with a flat line bundle.  One then has a moduli space parametrised by pairs $(\theta,\theta') \mapsto (\phi_{\theta}(L_f), \cL_{\theta'})$, where $\phi_{\theta}$ denotes the rotation by $\theta \in \bR / A\bZ$ and $\cL_{\theta'}$ denotes the local system with holonomy $\theta' \in \bC^{*}q^0 \subset \Lambda_{\bR}$.  One can formally allow line bundles with monodromy in $\Lambda_{\bR}^{*}$, and identify the  objects
\[
(\theta, q^{r} (a_0 + \sum a_i q^{t_i}))\  \equiv \ (\theta+r, a_0 + \sum a_i q^{t_i})
\]
so the valuation of the holonomy corresponds formally to translating the underlying Lagrangian by the corresponding flux, compare \cite[Section 4.1]{Au09} for the convergent case.  Translating all the way around the torus, this gives a moduli space  $\Lambda_{\bR}^*/\langle q^A\rangle$ of objects in the Fukaya category which are mirror to the points of the algebraic elliptic curve.
\end{rem}

\subsection{From quantum cohomology to Hochschild cohomology\label{subsec:naturalmap}}  The Hochschild complex of an $A_{\infty}$-category is defined in Section \ref{Subsec:BackgroundAlgebra}.  As explained in Section \ref{Subsec:BackgroundFukaya}, for a symplectic manifold $M$ there is a natural map from the quantum cohomology of $M$ to the Hochschild cohomology of its Fukaya category.

\begin{cor} \label{cor:natural_map_iso_2torus}
The natural map 
 \begin{equation} QH^*(T^2;\Lambda_{\bR}) \to HH^*(\sF(T^2)) \end{equation} is an isomorphism.
\end{cor}

\begin{proof}
The quantum cohomology $QH^*(T^2;\Lambda_{\bR}) \cong H^*(T^2;\Lambda_{\bR})$ is isomorphic to usual cohomology, hence has rank $4$ over $\Lambda_{\bR}$. By  Theorem \ref{thm:mirror_symmetry_elliptic}, Proposition \ref{rem:hh-invariant} and Corollary \ref{cor:rankHH}, $HH^*(\sF(T^2))$ is abstractly isomorphic to $H^*(T^2;\Lambda_{\bR})$ as an algebra over $\Lambda_{\bR}$, so it suffices to prove that the natural map is an injection on elements of degree $1$, i.e. that the images of the elements $PD[L_f], PD[L_s] \in QH^1(T^2)$ are distinct in  $HH^*(\sF(T^2))$.  In Lemma \ref{lem:HHrestrict}, we showed that upon identifying $QH^*(T^2;\Lambda_{\bR})$  with $H^*(T^2;\Lambda_{\bR})$ and $HF^*(L,L)$ with $H^*(L)$, the restriction map on ordinary cohomology agrees with the composition
\begin{equation} \label{eqn:subcategory}
QH^*(T^2) \rightarrow HH^*(\sF(T^2)) \rightarrow HF^*(L,L)
\end{equation}
where the second map is the projection to the $0$-th part of the filtration on the Hochschild complex, cf. Lemma \ref{lem:HHrestriction}.  We conclude that $PD[L_f]$ and $ PD[L_s]$ necessarily map to distinct elements, proving the result.
\end{proof}


\section{Quilts and generation\label{Sec:split}}


\subsection{Statement of the (split-)generation criterion}

Let $M$ be a symplectic manifold whose Fukaya category is well defined over some field $k$.  Typically this will be the Novikov field $\Lambda_{\bR}$ if $M$ is spin, or its characteristic 2 cousin otherwise, but since the choice plays no role in this section, we will keep to more general notation.  Let $\cA \subset \sF(M)$ be a full subcategory of the Fukaya category and write $\End(Tw(\cA))$ for the category of $A_{\infty}$-endofunctors of the category of twisted complexes over $\cA$. In this section, we use pseudo-holomorphic quilts to provide a criterion which is sufficient for $\cA$ to split-generate the entire Fukaya category.  In the next section, this criterion will be used to prove homological mirror symmetry for the four torus.

As explained in Section \ref{Subsec:BackgroundFukaya},  $\sF(M)$ is a cohomologically unital category which can be equipped with classes $e \in CF^{*}(K,K)$ representing the identity, essentially given by counts of rigid finite energy half-planes with boundary on $K$, cf. Lemma \ref{lem:units} and Remark \ref{rem:units=rigidplanes}.   Recall from Section \ref{Subsec:BackgroundAlgebra} the Yoneda $A_{\infty}$-embedding $\cY$ from $\sF(M)$ into the category of right $\sF(M)$ modules
\begin{align} \label{eq:yoneda} \cY \co \sF(M) & \longrightarrow \textrm{mod-}(\sF(M)) \\
K & \mapsto (L \mapsto CF^{*}(K,L)),
\end{align}
and the dual Yoneda embedding into the same category of modules
\begin{align} \cY^{\vee} \co \sF(M) & \longrightarrow \textrm{mod-}(\sF(M)) \\
K & \mapsto (L \mapsto \Hom(CF^{*}(L,K),k)). \end{align}

\begin{lem}\label{lem:poincareduality}
For each object $K$ there is an equivalence of $A_{\infty}$-modules

\begin{equation}   \cY(K) \to \cY^{\vee}(K)[n]. \end{equation}
\end{lem}
\begin{proof}
This is an $A_{\infty}$-extension of the familiar Poincar\'e duality statement (at the level of cohomology) $HF^i(L,K) \cong HF^{n-i}(K,L)^{\vee}$, and is a version of the assertion that the Fukaya category is a weak Calabi-Yau category (see \cite{Fukaya-cyclic}).  While we do not require these equivalences to be related to each other as $K$ varies, the existence of a natural transformation of functors $\cY \rightarrow \cY^{\vee}[n]$, as explained in \cite[Section 12j]{Seidel:FCPLT} proves the Lemma.
\end{proof}

Next, we observe that the construction of a twisted complex in \eqref{eq:tensor_product_chain_cmplex} by tensoring an object with a cochain complex requires only a unit satisfying the properties proved in Lemma \ref{lem:weak-unit-statement}.  Following Definition \ref{defin:projection-functor} we can therefore assign, to each pair $K_\pm \in \cA$, a ``projection" endofunctor 
\begin{equation} \cI (K_\pm): Tw(\cA) \rightarrow Tw(\cA) \end{equation}
which acts on objects by
\[ L \longrightarrow CF^{*}(K_-, L) \otimes K_{+}
\] 
and on morphisms by
\begin{align} 
\notag CF^{*}(L_{d-1},L_{d}) \otimes  \cdots \otimes CF^{*}(L_{0}, L_{1}) &  \longrightarrow   \\   & \!\!\!\!\!\!\!\!\!\!\!\!\!\!\!\!\!\!\!\!\!\!\!\!\!\!\!\!\!\!\!\!\!\!\!\!\!\!\!\!\!\!\!\!\!\!\!\!\!\!\!\!\!\!\!\!\!\!\!\!\!\!\!\!\!\!\!\!\!\!\!\!\!\!
\Hom(CF^{*}(K_-, L_0), CF^{*}(K_-, L_d)) \otimes CF^{*}(K_+, K_+)[1-d] \\ \label{eq:formula_yoneda_functor}
a_{d} \otimes \cdots \otimes a_{1} &  \mapsto \mu^{d+1}(a_{d}, \ldots, a_{1}, \cdot) \otimes e
\end{align}

The following theorem is the principal technical innovation of the paper.  The homological algebra applies in some generality -- only the Poincar\'e duality statement Lemma \ref{lem:poincareduality} is really essential -- but the geometric input from quilts currently constrains us to applying the result only in rather particular circumstances. 

\begin{thm} \label{thm:generation}
Suppose $\sF(M)$ is well-defined over the field $k$ and that the \emph{Axiom of Quilted Floer Theory} of Section \ref{Subsec:BackgroundQuilts} holds.  Suppose in addition:  
\begin{align*}
&  \parbox{30em}{$\bullet$ The natural map $QH^{*}(M;k) \to HH^*(\cA)$ is an isomorphism} \\
&  \parbox{30em}{$\bullet$ The identify $\id|_{\cA}$ lies in the category generated by the functors $\cI (K_{\pm}).$}
\end{align*}
Then the natural embedding $Tw(\cA) \subset Tw(\sF(M))$
is an equivalence.  In particular,  $\cA$ generates the Fukaya category.  Similarly,  if the second condition is replaced by 
\begin{align*}
&  \parbox{30em}{$\bullet$ The identify $\id|_{\cA}$ lies in the category split-generated by the functors $\cI (K_{\pm})$}
\end{align*}
then $Tw^{\pi}(\cA) \subset Tw^{\pi}(\sF(M))$ is an equivalence  and $\cA$ split-generates the Fukaya category.
\end{thm}
\begin{rem}  
Recall from Proposition \ref{rem:hh-invariant} that Hochschild cohomology is invariant under idempotent completion, hence the first hypothesis is unchanged in the two versions of the Theorem.
The reader may easily check that the proof we give implies that $\cA(M)$ generates the category $\cF^{\#}(M)$ in the appropriate sense.  In particular, one can interpret this result as giving a sufficient criterion for the difference between the Fukaya category and the category of correspondences  $\cF^{\#}(M)$ to disappear after passing to twisted complexes (or the idempotent completion thereof).
\end{rem}


\subsection{Functor isomorphisms from quilts}
Let us introduce the subcategory
\begin{equation} \cA^{\oplus} (M) \subset \sF^{\#}(M) \end{equation}
whose objects are those generalised Lagrangian correspondences
\begin{equation} \label{eqn:ourfavourites} \xymatrix{ \pt \ar[r]^{L} & M  \ar[r]^{L_{1,2}} & M  \ar[r]^{L_{2,3}} & \cdots  \ar[r]^{L_{k-2, k-1}}&  M \ar[r]^{L_{k-1,k}} & M} \end{equation}
where every Lagrangian $L_{i,i+1}$ is of the form 
\begin{equation} L^{-}_{i} \times L_{i+1} \subset M^{-} \times M\end{equation}
with both $L_{i}$ and $L_{i+1}$ being objects of $\cA(M)$.  The first result is:

\begin{lem}  \label{lem:equal}
The inclusion $\cA \subset \cA^{\oplus}$ induces an equivalence
\begin{equation} \Tw(\cA) \cong \Tw(\cA^{\oplus}). \end{equation} 
In particular, given  a length $2$ sequence
\begin{equation} \xymatrix{ \pt \ar[r]^{L} & M  \ar[r]^{\bfK} & M } \end{equation}
where $\bfK = K^{-}_{-} \times K_{+}$, there is an equivalence
\begin{equation} \label{eq:isomorphism_sequence} \left( \xymatrix{ \pt \ar[r]^{L} & M  \ar[r]^{\bfK} & M} \right) \cong CF^{*}(K_-, L)  \otimes K_+.\end{equation}
\end{lem}
\begin{proof}
It suffices to prove that every generalised Lagrangian correspondence of the shape (\ref{eqn:ourfavourites}) is equivalent to a twisted complex of objects of $\cA$, since certainly $\cA^{\oplus}$ generates Tw\,$\cA^{\oplus}$.  The category $\cA^{\oplus}$ has a nice filtration by the length of the generalized correspondence, and the proof will be by induction. For simplicity of notation, we just do the ``next-to-base" case of the length $2$ sequence  in Equation \eqref{eq:isomorphism_sequence}.

That the right hand side is a twisted complex is clear since $CF^{*}(K_-, L)$ is a finite dimensional $\bZ$-graded vector space.  To prove the existence of the equivalence, we first compute
 that
 \begin{equation}
 \begin{aligned} CF^{*} & \left( \xymatrix{ \pt \ar[r]^{L} & M  \ar[r]^{\bfK} & M } ,  CF^{*}(K_-, L)  \otimes K_+ \right)
  \\ &  \qquad = CF^{*}(K_-, L) \otimes CF^{*}( L \times   K_+^-, K_- \times K_+^-). 
 \end{aligned}
 \end{equation}
To check this, recall from Remark \ref{rem:transposeLags} that to compute Floer chains between generalized Lagrangians, we transpose the second sequence (to give the ordered triple $(L, \bfK, K_+^-)$ up to the algebraic factor $CF(K_-, L)$ which is just a graded vector space), then take the constituent Lagrangians in alternating pairs (giving $L \times K_+^-$ and $\bfK$ both viewed in $M^- \times M$),  and then take Floer cochains; here we can split the resulting Floer complex using the K\"unneth theorem, since we are working over a field.
At the level of homology, and keeping track of the degrees, the above group becomes 

\begin{equation}
\begin{aligned}  HF^{*}(K_-, L) & \otimes HF^{*}( L, K_-) \otimes  HF^{*}( K_+, K_+)  \cong \\ & \End( HF^{*}( K_-, L)) \otimes HF^{*}( K_+, K_+), \end{aligned} \end{equation}
using the duality isomorphism $HF^{*}( L, K_-) \cong HF^{n-*}(K_-, L)^{\vee} $.  A chain level representative of the tensor product of the identity on $HF^{*}( K_-, L)$ with the identity of $HF^{*}(K_+, K_+)$ is the desired quasi-isomorphism.
\end{proof}

The $A_{\infty}$-structure on the group of endomorphisms of 
\begin{equation} CF^{*}(K_-, L)  \otimes K_+  \end{equation} 
is explicitly given by the tensor product of the $dg$-algebra $ \End( CF^{*}( K_-, L))$ with the $A_{\infty}$-algebra $CF^{*}( K_+, K_+)$, with higher products given by
\begin{equation} \mu^{d}(A_{d} \otimes x_{d}, \cdots , A_{1} \otimes x_{1}) = (-1)^{\triangle} \left( \prod A_{i} \right) \otimes \mu^{d}( x_{d}, \cdots, x_{1}). \end{equation}
whenever each matrix $A_{i}$ is of pure degree, with the sign equal to the expression
\begin{equation*} \triangle = \sum_{i < j} |A_i| ( |x_j| +1 ). \end{equation*}

\begin{lem}
Let $K_-, K_+ \subset M$ be Lagrangian submanifolds and let $\bfK$ denote the product $K_-^{-}\times K_+$. For appropriate choices of almost complex structures and perturbation data, there is a canonical isomorphism of $A_{\infty}$-algebras \label{Small change in statement}
\begin{multline} CF^{*} \left( \xymatrix{ \pt \ar[r]^{L} & M  \ar[r]^{\bfK} & M } ,\xymatrix{ \pt \ar[r]^{L} & M  \ar[r]^{\bfK} & M }  \right) \\  \cong  \End(CF^{*}(K_-, L) )  \otimes  CF^{*}( K_+, K_+).\end{multline}
\end{lem}

\begin{proof}
As recalled in Section \ref{Subsec:BackgroundQuilts}, Floer complexes in quilt theory are defined using auxiliary choices of small Hamiltonians to perturb the Lagrangians $ L \subset M$ and $K^-_- \times K_+ \subset M^{-} \times M$, along with generic families of almost complex structures on $M$ parametrised by points of the abstract quilt.  In our situation, however, one can work with a much smaller set of perturbations, namely those used to define the differential on $CF^{*}(K_-, L)$ and the $A_{\infty}$-structure on $CF^*(K_+, K_+)$, as we now explain.

Concretely, the first step in defining  $CF^{*}(K_-, L)$ is to pick a generic Hamiltonian $H(K_-, L)$ and set the generators to be chords of the Hamiltonian flow starting at $K_-$ and ending at $L$. Similarly, one makes a choice of $H(K_+)$ to define $CF^{*}(K_+, K_+)$.  We consider the function (note the signs!)
\begin{equation} H(L, \bfK) \equiv -H(K_-, L) \oplus -H(K_+) \oplus H(K_-,L) \co M \times M^- \times M \to \bR .\end{equation}
The genericity assumption on each factor implies that the image of $L \times K_{+}^- \times K_{-} $ under the time-$1$ Hamiltonian flow of  $H(L, \bfK)$ is transverse to $K_{-} \times K_{+}^- \times L$.  Unwinding the definition, 
\begin{equation} \label{eqn:quiltgroup} CF^{*} \left( \xymatrix{ \pt \ar[r]^{L} & M  \ar[r]^{\bfK} & M } ,\xymatrix{ \pt \ar[r]^{L} & M  \ar[r]^{\bfK} & M }  \right) \end{equation}
is just the Floer chain group for these Lagrangian submanifolds $L \times K_+^- \times K_-$ and $K_- \times K_+^- \times L$ of $M\times M^- \times M$, cf. Remark \ref{rem:transposeLags} and the preceding discussion.  In practice, it is therefore the group generated by triples of chords: 
\begin{enumerate}
\item from $L$ to $K_-$ along the Hamiltonian flow of $-H(K_-, L)$, 
\item  from $K_+$ back to $K_+$ along the flow of $H(K_+)$ and 
\item from $K_-$ to $L$ along the flow of $H(K_-, L)$. 
\end{enumerate} 
Note that the sign on the middle Hamiltonian has changed since we are now considering the Lagrangians in $M$ rather than in $M^-$.  We can represent this set-up graphically as in Figure \ref{fig:quilt_chords_conventions}, where we also reverse the sign on the first Hamiltonian by thinking of the flow as going from $K^-$ to $L$.  We do not label the Hamiltonians chords in Figure \ref{fig:quilt_chords_conventions} to reiterate the fact that, keeping in mind the direction of the arrows, they are exactly the same Hamiltonians that appeared in the construction of the category $\scrF(M)$.
\begin{figure}[h]
   \centering
    \begin{picture}(0,0)%
\epsfig{file=quilt_chords_conventions.pstex}%
\end{picture}%
\setlength{\unitlength}{3947sp}%
\begingroup\makeatletter\ifx\SetFigFont\undefined%
\gdef\SetFigFont#1#2#3#4#5{%
  \reset@font\fontsize{#1}{#2pt}%
  \fontfamily{#3}\fontseries{#4}\fontshape{#5}%
  \selectfont}%
\fi\endgroup%
\begin{picture}(3892,445)(-74,-1319)
\put(3526,-1261){\makebox(0,0)[lb]{\smash{{\SetFigFont{12}{14.4}{\familydefault}{\mddefault}{\updefault}{\color[rgb]{0,0,0}$L$}%
}}}}
\put(901,-1261){\makebox(0,0)[lb]{\smash{{\SetFigFont{12}{14.4}{\familydefault}{\mddefault}{\updefault}{\color[rgb]{0,0,0}$(K_-,K_+)$}%
}}}}
\put(-74,-1261){\makebox(0,0)[lb]{\smash{{\SetFigFont{12}{14.4}{\familydefault}{\mddefault}{\updefault}{\color[rgb]{0,0,0}$L$}%
}}}}
\put(2101,-1261){\makebox(0,0)[lb]{\smash{{\SetFigFont{12}{14.4}{\familydefault}{\mddefault}{\updefault}{\color[rgb]{0,0,0}$(K_+, K_-)$}%
}}}}
\end{picture}%
\caption{}
   \label{fig:quilt_chords_conventions}
\end{figure}

Having set up the chain complex, we should now consider the necessary choices of almost complex structures and perturbations of the Cauchy-Riemann equations.  For definiteness, consider the case of the product $\mu_{\sF^{\#}}^{2}$; the argument for the higher products is only notationally different.  The quilt controlling this product is shown in Figure \ref{fig:quilt_pair_pants}, where the correspondence at every seam is given by $K_-^{-} \times K_+$.
\begin{figure}[h]
   \centering
   \begin{picture}(0,0)%
\epsfig{file=quilt_pair_pants.pstex}%
\end{picture}%
\setlength{\unitlength}{3947sp}%
\begingroup\makeatletter\ifx\SetFigFont\undefined%
\gdef\SetFigFont#1#2#3#4#5{%
  \reset@font\fontsize{#1}{#2pt}%
  \fontfamily{#3}\fontseries{#4}\fontshape{#5}%
  \selectfont}%
\fi\endgroup%
\begin{picture}(2222,2046)(2326,-3382)
\put(3959,-2282){\makebox(0,0)[lb]{\smash{{\SetFigFont{12}{14.4}{\familydefault}{\mddefault}{\updefault}{\color[rgb]{0,0,0}$L$}%
}}}}
\put(2736,-2271){\makebox(0,0)[lb]{\smash{{\SetFigFont{12}{14.4}{\familydefault}{\mddefault}{\updefault}{\color[rgb]{0,0,0}$L$}%
}}}}
\end{picture}%
 \caption{}
   \label{fig:quilt_pair_pants}
\end{figure}
Since the correspondence decomposes as a product, we simply have four holomorphic maps to $M$, three of which are strips with boundaries on $L$ and $K_-$, and the remaining one of which is a pair-of-pants with boundary on $K_+$. Moreover, the count of rigid objects is given by requiring that each individual component be rigid.

For the pair-of-pants, we choose exactly the perturbations and almost complex structures used to define $\mu_{\sF}^2$ on the self-Floer homology of $K_+$, while for the strips we choose the family of almost complex structures used to define the differential on $CF^{*}(K^-, L)$, and no additional perturbation.   A little thought shows that the requirements that the strips be rigid forces their inputs and outputs to be the same chord; otherwise the relevant strip would move in (at least) a one-dimensional moduli space given by $\bR$-translations.  The counts of strips therefore essentially compute the continuation maps for a constant family of Hamiltonians.

Using the identification \begin{equation} \label{eq:identifynow} \End(CF^{*}(K_-, L)) \cong CF^{*}(K_-, L)^{\vee} \otimes CF^{*}(K_-, L) \end{equation}
we see that our count of strips is exactly realising the fact that, with respect to a choice of basis $\{a_j\}$, the product of elements of $CF^{*}(K_-, L)^{\vee} \otimes CF^{*}(K_-, L)$ in its $dg$-algebra structure is given by 

\begin{equation} \mu_{2} (a_{k}^{\vee} \otimes a_{l}, a_{i}^{\vee} \otimes a_{j}) = \delta_{il} a_{k} \otimes a_{j}^{\vee}. \end{equation}

Although we have not discussed coherent orientations in any detail, the signs associated to quilted Riemann surfaces by Wehrheim and Woodward in \cite[Section 4]{WW:orientquilts} arise from determinant line bundles associated to families of Cauchy-Riemann operators \emph{which can be deformed into split operators} when viewing the quilt as a single map into a product space.  The identification of generators of the complex of Equation \ref{eqn:quiltgroup} with triples of chords in $M$, and of the quilted map of Figure  \ref{fig:quilt_pair_pants} with a tuple of maps into $M$ itself, therefore gives a natural identification of the determinant line on this particular moduli space of quilts with a tensor product of determinant lines associated to the constituent domains of the quilt.   From this it follows directly that the identification of Equation \ref{eq:identifynow} is compatible with signs, which completes the construction of the $A_{\infty}$-isomorphism. 
\end{proof}

This argument readily generalizes to multiple Lagrangians, and proves that the images of $\cA(M)$ under the functors $\Phi(\bfK)$ and  $\cI (K_\pm)$ are isomorphic subcategories  of $\Tw \cA^{\oplus}(M)$.   In particular, for all pairs of Lagrangians $L_0$ and $L_1$, we have an isomorphism
\begin{multline} \label{eq:endomorphism_algebra_correspondence_iso_twisted_complex} CF^{*} \left( \xymatrix{ \pt \ar[r]^{L_0} & M  \ar[r]^{\bfK} & M } ,\xymatrix{ \pt \ar[r]^{L_1} & M  \ar[r]^{\bfK} & M }  \right) \\  \cong  \Hom(CF^{*}(K_-, L_0),CF^{*}(K_-, L_1) )  \otimes  CF^{*}( K_+, K_+).\end{multline}  
We now prove that this isomorphism is strong enough to imply that 

\begin{lem} \label{lem:isomorphic_functors}
Given a pair of Lagrangians $K_-$ and $K_+$ in $\cA(M)$, the functors $\Phi(\bfK)$ and $\cI (K_\pm) \in nu$-$fun(\Tw\cA^{\oplus}, \Tw\cA^{\oplus})$ are isomorphic.
\end{lem}
\begin{proof}
Given a sequence of Lagrangians $\{ L_{d}, \cdots, L_{0} \}$, Mau, Wehrheim, and Woodward use holomorphic quilts to  define a map 
\begin{multline} \label{eq:multilinear_term_MWW_functor}
CF^{*}(L_{d-1},L_{d}) \otimes  \cdots \otimes CF^{*}(L_{0}, L_{1})  \to \\ 
\ CF^{*} \left( \xymatrix{ \pt \ar[r]^{L_0} & M  \ar[r]^{\bfK} & M } ,\xymatrix{ \pt \ar[r]^{L_{d}} & M  \ar[r]^{\bfK} & M }  \right)[1-d]. \end{multline}
On the other hand, in Equation \eqref{eq:formula_yoneda_functor}  we defined a map with the same source, but with target 
\begin{equation*} \Hom(CF^{*}(K_-, L_0), CF^{*}(K_-, L_d)) \otimes CF^{*}(K_+, K_+)[1-d]. \end{equation*}
An isomorphism between these two targets is given  by Equation \eqref{eq:endomorphism_algebra_correspondence_iso_twisted_complex}, and our goal is to show that the maps \eqref{eq:formula_yoneda_functor} and \eqref{eq:multilinear_term_MWW_functor} are intertwined by this isomorphism.

This is again an immediate consequence of the fact that our choices of perturbation on $M^- \times M$ are induced from perturbations on $M$, and the fact that the unit in $CF^{*}(K_+, K_+)$ is given by a count of perturbed pseudo-holomorphic planes with boundary on $K_+$, cf. Lemma \ref{lem:units}.  For example, the case $d=3$ is controlled on the quilt side by Figure \ref{fig:quilt_third_functor}, where the dark region corresponds to the unit of $CF^{*}(K_+, K_+)$, and the light region is a holomorphic disc with $5$ punctures, realising a $\mu_{\sF}^{4}$ product.
\begin{figure}    \centering
\begin{picture}(0,0)%
\epsfig{file=quilt_third_functor.pstex}%
\end{picture}%
\setlength{\unitlength}{3947sp}%
\begingroup\makeatletter\ifx\SetFigFont\undefined%
\gdef\SetFigFont#1#2#3#4#5{%
  \reset@font\fontsize{#1}{#2pt}%
  \fontfamily{#3}\fontseries{#4}\fontshape{#5}%
  \selectfont}%
\fi\endgroup%
\begin{picture}(1725,2325)(2776,-3136)
\put(4046,-1603){\makebox(0,0)[lb]{\smash{{\SetFigFont{12}{14.4}{\familydefault}{\mddefault}{\updefault}{\color[rgb]{0,0,0}$L_3$}%
}}}}
\put(3001,-1561){\makebox(0,0)[lb]{\smash{{\SetFigFont{12}{14.4}{\familydefault}{\mddefault}{\updefault}{\color[rgb]{0,0,0}$L_0$}%
}}}}
\put(3226,-2311){\makebox(0,0)[lb]{\smash{{\SetFigFont{12}{14.4}{\familydefault}{\mddefault}{\updefault}{\color[rgb]{0,0,0}$L_1$}%
}}}}
\put(3901,-2311){\makebox(0,0)[lb]{\smash{{\SetFigFont{12}{14.4}{\familydefault}{\mddefault}{\updefault}{\color[rgb]{0,0,0}$L_2$}%
}}}}
\end{picture}%
 \caption{}
   \label{fig:quilt_third_functor}
\end{figure}
\end{proof}

Next, we consider the category $\cA(M^{-} \times M)$ consisting of Lagrangians in $M^{-} \times M$ of the form $K^{-}_{-} \times K_+$ where $K_{\pm}$ lie in $\cA(M)$.

\begin{lem} \label{lem:mwwisomorphism}
The Mau-Wehrheim-Woodward functor of Section \ref{Subsec:BackgroundQuilts}
\begin{equation}
\sF(M^{-} \times M) \to \End(\sF^{\#}(M)) \end{equation}
restricts to a fully faithful $A_{\infty}$-functor
\begin{equation}\Phi \co \cA(M^{-} \times M) \to \End(\cA^{\oplus}(M)) \end{equation}
\end{lem}
\begin{proof}
That $ \cA(M^{-} \times M)$ preserves the subcategory $\cA^{\oplus}(M) \subset \sF^{\#}(M)$ is obvious, so we'll prove that the resulting functor is fully faithful.    If $\bfL = L^{-}_- \times L_{+}$ is another object of $ \cA(M^{-} \times M)$, we must prove that the Mau-Wehrheim-Woodward map
\begin{equation} CF^{*}(\bfK, \bfL) \to \hom_{\scrQ}(\Phi(\bfK), \Phi(\bfL) )
\end{equation}
is a quasi-isomorphism, where $ \hom_{\scrQ}$ denotes the chain complex of $A_{\infty}$-natural transformations.  To check this, we first prove that the natural map
\begin{equation} H\hom_{\scrQ} (\Phi(\bfK), \Phi(\bfL) ) \to \hom_{H(\scrQ)} (H\Phi(\bfK), H\Phi(\bfL) )
\end{equation}
taking a homology class of $A_{\infty}$-natural transformations to the corresponding natural transformation of the Wehrheim-Woodward functors on the Donaldson category, is an isomorphism.  By Lemma \ref{lem:isomorphic_functors},  it suffices to prove the analogous result for $\cI(K_\pm)$ and $\cI(L_{\pm})$.  Recall from Lemma
\ref{lem:yoneda_for_projection} that we have an isomorphism \begin{equation} H^* \hom_{\scrQ}(\cI(K_\pm), \cI(L_\pm)) \cong \hom_{H\scrQ}( H^* \cI(K_\pm),  H^* \cI(L_\pm)). \end{equation}
In particular, every homology class of  $A_{\infty}$-natural transformations from $\Phi(\bfK)$ to $\Phi(\bfL)$ can be detected at the level of  Donaldson categories.  Passing to cohomology, it finally suffices to prove that the Wehrheim-Woodward functor 
\begin{equation} HF^{*}(\bfK, \bfL) \to  \hom_{H\scrQ}(H\Phi(\bfK), H\Phi(\bfL) ) \end{equation}
is an isomorphism.  Both groups can be naturally identified with 

\begin{equation}  HF^{*}(K_-,L_-) \otimes HF^{*}(K_+, L_+) \end{equation}
and with respect to these natural identifications the map between them is the identity.
\end{proof}

\subsection{Incorporating the diagonal}
Before proceeding, we recall some elementary facts about Floer complexes and changing sign of the symplectic form.  Let $(M, \omega)$ be any symplectic manifold and pick a Hamiltonian function $H: M \rightarrow \bR$.  For Lagrangian submanifolds $K_-,K_+$ in $M$ we have defined the Floer complex
\begin{equation} CF^{*}(K_-, K_+) \end{equation}
to be generated by time-$1$ Hamiltonian chords of $H$.    If we pass to $(M^{-}, -\omega)$, then $H^{-} = H$ still defines a Hamiltonian function.  For simplicity, we assume below that there is a fixed almost complex structure $J$ and time-independent Hamiltonian function on $M$ controlling all holomorphic curves (the case of varying $J$ and $H$ is conceptually the same but notationally more involved). We equip  the product $M^{-} \times M$  with the almost complex structure $-J \oplus J$ and the Hamiltonian $(H^{-} \oplus H)/2$.  

\begin{lem} \label{lem:changesign}
There are natural isomorphisms of chain complexes
\begin{enumerate} 
\item $CF^{*}(K_-^{-}, K_+^{-}) \cong CF^{*}(K_+, K_-)$
\item  $CF^{*}(\Delta, K_-^{-} \times K_+) \cong CF^{*}(K_-, K_+).$
\end{enumerate}
\end{lem}
\begin{proof}
It is obvious that chords of $H$ in $M^{-}$ from $K_-^{-}$ to $K_+^{-}$ correspond exactly to chords in $M$ from $K_+$ to $K_-$, so $CF^{*}(K_-^{-}, K_+^{-}) $ and $CF^{*}(K_+, K_-)$ have canonically identified bases.  We must now stare at the differential.   Consider the almost complex structure $-J$ on $M$ which obviously tames $-\omega$ if $J$ tames $\omega$.

Given chords $x$ and $y$ from $K_+$ to $K_-$ and a $J$-holomorphic map $u \co Z = \bR \times [0,1] \to M$ with boundary on $K_-$ and $K_+$, and asymptotic boundary conditions $x$ at $+\infty$ and $y$ at $-\infty$, the map $u^{-} \co Z \to M^{-}$ defined as the composition 
\begin{equation} (s,t) \mapsto (s,-t+1/2) \circ u \end{equation}
is $(-J)$-holomorphic, has boundary on $K_+^{-}$ and $K_-^{-}$ and asymptotic boundary conditions $x^{-}$ at $+\infty$ and $y^{-}$ at $-\infty$. This proves the first statement.

For the second claim, an $H^{-} \oplus H$ chord from the diagonal to $K_-^{-} \times K_+$ gives an $H/2$ chord from $K_-$ to a point $q \in M$ together with an $H/2$ chord from $q$ to $K_+$.  By concatenating the two, we obtain an $H$ chord from $K_-$ to $K_+$ as desired. Matching the differentials is a standard argument which involves ``splitting" a holomorphic strip in $M$ at its middle horizontal line. 
\end{proof}

Let us now consider the subcategory
\begin{equation} \cA^{\Delta}(M^{-} \times M) \subset \sF(M^{-} \times M)\end{equation}
whose objects are the diagonal, together with all objects of  $\cA(M^{-} \times M)$.   Again, the Mau-Wehrheim-Woodward construction gives a functor
\begin{equation} \label{eq:add_diagonal} \Phi \co \cA^{\Delta}(M^{-} \times M) \to \End(\cA^{\oplus}(M)). \end{equation}
Recall from Section \ref{Subsec:BackgroundAlgebra} that the Hochschild cohomology $HH^*(\sC)$ of an $A_{\infty}$-category $\sC$ is the group of $A_{\infty}$-natural transformations from the identity functor of $\sC$ to itself. Recall also the natural map $QH^*(M) \rightarrow HH^*(\scrF(M))$ from Equation \ref{eq:open-closed-cohomology}, which was obtained by counting holomorphic discs with an interior marked point constrained to pass through a given cycle in $M$ and boundary marked points lying on cycles on particular Lagrangians.   In order to relate that map to one coming from holomorphic quilts, we make a particular choice of  Hamiltonian perturbation to compute the Lagrangian Floer homology of the diagonal.  Consider the function
\begin{equation}
 \frac{1}{2} \left( H^{-}_{1-t}\oplus H_{t} \right) \co M^{-} \times M \to \bR
\end{equation}
for which time-$1$ Hamiltonian chords with endpoints on $\Delta$ are readily seen to agree with time-$1$ orbits of $H_t$, yielding a chain level isomorphism
\begin{equation}
SC^{*}(M)   \cong CF^{*}(\Delta, \Delta).
\end{equation}

\begin{figure}[h]
   \centering
\begin{picture}(0,0)%
\includegraphics{seidel-map-quilt.pstex}%
\end{picture}%
\setlength{\unitlength}{3947sp}%
\begingroup\makeatletter\ifx\SetFigFont\undefined%
\gdef\SetFigFont#1#2#3#4#5{%
  \reset@font\fontsize{#1}{#2pt}%
  \fontfamily{#3}\fontseries{#4}\fontshape{#5}%
  \selectfont}%
\fi\endgroup%
\begin{picture}(1633,2712)(1501,-5887)
\put(2251,-4786){\makebox(0,0)[lb]{\smash{{\SetFigFont{9}{10.8}{\familydefault}{\mddefault}{\updefault}{\color[rgb]{0,0,0}$R$}%
}}}}
\put(2176,-3286){\makebox(0,0)[lb]{\smash{{\SetFigFont{9}{10.8}{\familydefault}{\mddefault}{\updefault}{\color[rgb]{0,0,0}$a_0$}%
}}}}
\put(2926,-5836){\makebox(0,0)[lb]{\smash{{\SetFigFont{9}{10.8}{\familydefault}{\mddefault}{\updefault}{\color[rgb]{0,0,0}$a_1$}%
}}}}
\put(1576,-5836){\makebox(0,0)[lb]{\smash{{\SetFigFont{9}{10.8}{\familydefault}{\mddefault}{\updefault}{\color[rgb]{0,0,0}$a_3$}%
}}}}
\put(2251,-5836){\makebox(0,0)[lb]{\smash{{\SetFigFont{9}{10.8}{\familydefault}{\mddefault}{\updefault}{\color[rgb]{0,0,0}$a_2$}%
}}}}
\end{picture}%
   \caption{}
   \label{fig:seidel-map-quilt}
\end{figure}

\begin{lem} 
Under the above isomorphism, the ``open-closed" string map \eqref{eq:seidel-open-closed} agrees with the map induced by pseudo-holomorphic quilts; in particular, the following diagram commutes
\begin{equation}
\xymatrix{ HF^{*}(\Delta,\Delta) \ar[dr]^{MWW} \ar[rr]^{\sim} & & QH^{*}(M) \ar[dl] \\
& HH^{*}(\cF(M)) .& }
\end{equation}
\end{lem}
\begin{proof}[Sketch of Proof]
The functor defined by Mau-Wehrheim-Woodward  is controlled by the quilt displayed in Figure \ref{fig:seidel-map-quilt}.  Ignoring the choices of Hamiltonians (interpolating between them yields chain homotopies, and hence does not affect the computation at the level of homology), we note that the matching condition along the seam is the diagonal and that both regions depicted in Figure \ref{fig:seidel-map-quilt} are therefore mapping to $M$.  Erasing the seam from that figure, we obtain Figure \ref{fig:seidel-map}, so the count of such quilts agrees with the count of punctured pseudo-holomorphic discs in $M$.  That erasing the seam in this way is legitimate is a basic aspect of quilted Floer theory.
\end{proof}

Recall that either $M$ is spin or we work in characteristic 2, and the \emph{Axiom} implies that the diagonal $\Delta$ corresponds, under the Mau-Wehrheim-Woodward functor $\Phi$, to the identity functor of $\cA$.  We may therefore define the map $QH^*(M) \rightarrow HH^*(\cA)$ by composing the PSS isomorphism from $QH^{*}(M)$ to $HF^{*}(\Delta, \Delta)$ with the MWW functor $\Phi$  landing in the group of natural transformations of the identity functor.  Alternatively, the preceding Lemma shows that this map agrees on homology with the ``open-closed string map".

\begin{lem} \label{lem:isowithdiagonal}
If the map $QH^{*}(M) \to HH^*(\cA)$ is an isomorphism then \eqref{eq:add_diagonal} is a fully faithful embedding.\end{lem}
\begin{proof}
 By Lemma \ref{lem:mwwisomorphism} it suffices to prove that for every product Lagrangian $\bfK$, $\Phi$ induces quasi-isomorphisms
\begin{align*} CF^*(\Delta, \bfK) & \to \hom_{\scrQ}( \id , \Phi( \bfK)) \\
CF^*( \bfK, \Delta) & \to \hom_{\scrQ}( \Phi( \bfK),  \id).
 \end{align*} 
By an analogous argument to that used in the proof of Lemma \ref{lem:mwwisomorphism}, one proves that one can pass to the  cohomological category without losing any information, and that there are natural isomorphisms
\begin{align} 
HF^*(\Delta, \bfK) & \cong HF^{*}(K_-, K_+) \cong H^{*}  \hom_{\scrQ}( \id , \cI(K_\pm) )\\ 
HF^{*}( \bfK , \Delta) & \cong HF^{*}(K_+, K_-) \cong H^{*} \hom_{\scrQ}(  \cI(K_\pm) , \id ).
\end{align}
Here the first isomorphisms in each line are provided by Lemma \ref{lem:changesign}.  For the subsequent isomorphisms, we  explicitly compute that
\begin{equation} \hom_{\scrQ}( \id , \cI(K_\pm) ) \cong \hom_{\textrm{mod}-\cA}(\cY(K_{-}), \cY(K_{+})) \end{equation}
whilst, appealing moreover to Lemma \ref{lem:poincareduality},
\begin{equation}  
\begin{aligned} 
\hom_{\scrQ}(  \cI(K_\pm) , \id ) & \cong  \hom_{\textrm{mod}-\cA}(\cY^{\vee}(K_{+}), \cY(K_{-})) \\ & \cong \hom_{\textrm{mod}-\cA}(\cY(K_{+})[n], \cY(K_{-})) .
\end{aligned}
\end{equation}
In both cases the Yoneda Lemma \ref{lem:Yoneda} now implies the desired result.
\end{proof}

\begin{proof}[Proof of Theorem \ref{thm:generation}]
This is now completed as follows.  We have a diagram
\[
\begin{array}{ccc}
\cA(M^- \times M) & \stackrel{\Phi}{\longrightarrow} & \End(\cA^{\oplus}(M)) \\
\downarrow & & \parallel \\
\cA^{\Delta}(M^- \times M) & \stackrel{\Phi}{\longrightarrow} & \End(\cA^{\oplus}(M))
\end{array}
\]
with the left vertical map being a fully faithful embedding. The hypotheses of Theorem \ref{thm:generation} imply, via Lemma \ref{lem:isowithdiagonal}, that the two horizontal arrows are fully faithful embeddings which have the same image.  Since on the lower line $\Delta \mapsto \id_{\cA^{\oplus}}$, we can express the identity functor of $\cA^{\oplus}$ as (a summand in) some iterated cone amongst objects of $\cA(M^-\times M)$.  On the other hand, the Mau-Wehrheim-Woodward construction, restricted to a subcategory of $\cF(M^- \times M)$ gives a functor
\[
\cA^{\Delta}(M^- \times M) \longrightarrow \End(\sF^{\#}(M))
\]
under which $\Delta \mapsto \id_{\sF^{\#}}$.  Since this latter functor is exact, it takes exact triangles to exact triangles, and idempotents to idempotents.  It follows that $\id_{\sF^{\#}}$ can be expressed as (a summand in) an iterated cone amongst objects lying in the image of $\cA(M^- \times M)$.  This implies that objects of $\cA$ resolve the diagonal in $\sF(M)$, which immediately implies that they (split-)generate that category.
\end{proof}

\begin{rem}
Note that the second hypothesis in the statement of Theorem \ref{thm:generation} implies only that $\id_{\cA^{\oplus}}$ is generated by objects of $\cA$; to generate the identity functor of $\sF$ or $\sF^{\#}$ we pass back and forth into geometry and appeal to the special role played by the diagonal. \end{rem}

\section{Homological mirror symmetry for the 4-torus\label{sec:HMS}}

We fix a \emph{symplectic} splitting $T^4 = (T^2)^-\times T^2$, noting that $T^2$ has an orientation reversing involution.  Recall the Lagrangian submanifolds $L_f, L_s \subset T^2$.  By taking all possible products of pairs of these, we obtain a subcategory $\cA(T^2\times T^2)$ which comes equipped, via quilts and Lemmata \ref{lem:equal} and \ref{lem:mwwisomorphism}, with a fully faithful functor
\[
\cA((T^2)^-\times T^2) \ \longrightarrow \ \End(\Tw^{\pi}(\cA(T^2)).
\]
Using mirror symmetry for the 2-torus factors, Theorem \ref{thm:mirror_symmetry_elliptic}, we know that $Tw^{\pi}(\cA(T^2)) \simeq D^b_{\infty}(E)$ are quasi-isomorphic $A_{\infty}$-categories.  By the result of To\"en, Theorem \ref{thm:toen}, we view the functor above as a fully faithful functor
\begin{equation} \label{eqn:functor_on_product}
\cA((T^2)^-\times T^2) \ \longrightarrow \ \End(D^b_{\infty}(E)) \simeq D^b_{\infty}(E\times E).
\end{equation}

\begin{lem} \label{lem:image_generates_on_product}
The image of the functor (\ref{eqn:functor_on_product}) split-generates the derived category. Hence, there is an equivalence $\Tw^{\pi}(\cA(T^2\times T^2)) \simeq D^b_{\infty}(E\times E)$.
\end{lem} 

\begin{proof}
We again appeal to the theorem of Orlov \cite{Orlov} asserting that on an algebraic variety $Z$ of dimension $n$ over an algebraically closed field $k=\bar{k}$, the consecutive powers $(\cE^{\otimes 0}, \cE^{\otimes 1}, \ldots, \cE^{\otimes n})$ of a very ample line bundle $\cE$ suffice to split-generate the derived category.  

By construction  $L_s \times L_s$ goes over to (a quasi-representative of) the structure sheaf $\mathcal{O}_{E\times E}$, and $L_{f} \times L_{f}$ to the structure sheaf of a point.  The Lagrangians  $L_f \times L_s$ and $L_s \times L_f$ map, respectively to the structure sheaves of divisors $(\{pt\}\times E)$ and  $(E \times \{pt\})$.  Since $\cO$ and $\cO_p$ generate the derived category of the elliptic curve, cf. the proof of Theorem \ref{thm:mirror_symmetry_elliptic}, from $\cO \boxtimes F$ and $\cO_p\boxtimes F$ one obtains $G \boxtimes F$ for any sheaf $G$ on $E$;  repeating for the other factor, the image of the functor contains all sheaves $G\boxtimes G'$ on $E \times E$.  In particular, we have $\cO(D) \boxtimes \cO(D')$ for all divisors $D$ and $D'$. Choosing $D$ and $D'$ suitably, such exterior tensor products realise arbitrary powers of a very ample line bundle. 

The Lemma now follows from Remark \ref{rem:split-generate}.
\end{proof}

\begin{cor} \label{cor:hyp1}
The identity functor of $\cA(T^2\times T^2)$ is split-generated by the projection functors $\cI(K_{\pm})$ associated to pairs $K_{\pm} \in \cA(T^2\times T^2)$.
\end{cor}

\begin{proof}
Iterating the above argument, we have a fully faithful functor
\[
\cA(T^4\times T^4) \rightarrow \End(D^b_{\infty}(E\times E)) \simeq D^b_{\infty}(E^4)
\]
whose image contains the powers $\cE^{\otimes n}$ of a very ample line bundle on $E^4$.  In particular, the identity functor of $D^b_{\infty}(E\times E)$ can be split-generated by the projection functors associated to iterated products, i.e. pairs $K_{\pm} \in \cA(T^2\times T^2)$, since the diagonal $\Delta_{E\times E} \in D^b_{\infty}(E^4)$ is split-generated by the $\cE^{\otimes n}$.
\end{proof}  
   
\begin{rem} \label{rem:go_to_4th_product}
Note that to deduce mirror symmetry for $T^4$ from mirror symmetry for $T^2$, we must resolve the diagonal $T^4 \subset T^4\times T^4$, which is why we eventually encounter an abelian 4-fold $E^4$.
\end{rem}
  
\begin{cor} \label{cor:natural_map_iso_4torus}
The natural map 
 \begin{equation} \label{eq:map_quatum_HH_T4} QH^*(T^4;\Lambda_{\bR}) \to HH^*(\cA(T^2\times T^2)) \end{equation} is an isomorphism.
\end{cor}

 \begin{proof}
 According to Lemma \ref{lem:image_generates_on_product}, and appealing to Corollary \ref{cor:rankHH}, the rings $QH^*(T^4)$ and $HH^*(\cA(T^2\times T^2))$ are abstractly isomorphic. Composing \eqref{eq:map_quatum_HH_T4} with the restriction from $  HH^*(\cA(T^2\times T^2)) $  to the cohomology of the linear Lagrangian tori which form the objects of $\cA(T^2\times T^2) $, we conclude the desired result, by exactly the same proof used for the corresponding result for the 2-torus in Corollary \ref{cor:natural_map_iso_2torus}.
   \end{proof}

\begin{thm} \label{thm:mirror_symmetry_4torus}
Let $T^4$ denote the standard symplectic 4-torus and $E$ the Tate elliptic curve of Definition \ref{defin:Tate}.  There is an equivalence of triangulated categories, linear over the Novikov field $\Lambda_{\bR}$,
\[
D^{\pi}\sF(T^4) \simeq D^b(E\times E).
\]
\end{thm}

\begin{proof}
Corollary \ref{cor:hyp1} and Corollary \ref{cor:natural_map_iso_4torus}  verify both the criteria of Theorem \ref{thm:generation}. It follows that the fully faithful embedding of Equation \ref{eqn:functor_on_product} induces the claimed equivalence on descending to cohomology.
\end{proof}


\section{Classification of genus 2 objects}

\subsection{Digression} Before addressing Theorem \ref{thm:classify}, we give a brief digression to point out that one can directly geometrically classify Maslov zero Lagrangian subtori in $(T^{2n}, \omega_{st})$ up to isomorphism in the split-closed derived Fukaya category: they are all isomorphic to linear Lagrangian tori\footnote{Polterovich has pointed out to us that part of the argument was known to Arnol'd \cite{Arnold:firststeps}.}.

\begin{prop} A Lagrangian torus $L \subset T^{2n}$ with Maslov class zero is Floer cohomologically indistinguishable from a linear Lagrangian torus.
\end{prop}

\begin{proof}
If $L\subset T^{2n}$ is a Lagrangian torus, the inclusion is $\pi_1$-injective if and only if $[L]\neq 0\in H_n(T^{n};\bZ)$.  Supposing this holds, the underlying primitive homology class $[L]_{prim}$ defines a linear Lagrangian torus $R \subset T^{2n}$.  $L$ lifts as a closed Maslov zero Lagrangian submanifold to the covering space  $T^*R$ of $T^{2n}$ defined by $\pi_1(R)$, which is easily seen to be naturally symplectomorphic to $T^*T^n$ (this follows from the transitivity of $Sp_{2n}(\bZ)$ on integral linear Lagrangian subspaces in $\bR^{2n}$).  It is now classical \cite{Arnold:firststeps} that $L$ must be homologous to the zero-section, so $[L]=[L]_{prim}$ was in fact primitive, and indeed from \cite{FSS2} we know that $L$ is isomorphic in $D^{\pi}\sF(T^*R)$ to the zero-section.   In fact, the latter result does not rely on the Maslov class, so we see that \emph{any homologically essential Lagrangian torus is quasi-isomorphic to a linear Lagrangian torus} and \emph{a posteriori} has vanishing Maslov class. We remind the reader that the argument of \cite{FSS2} involves embedding the category $\sF(T^*R)$ into a category of modules over the algebra of chains on the based loop space $C_*(\Omega R) \simeq \bZ[\pi_1(R)]$, by considering the association $K \mapsto WF^*(K,T_x^*R)$ to a Lagrangian submanifold $K$ of its ``wrapped" Floer cochain complex with the cotangent fibre.  Since $R$ is an Eilenberg-MacLane space, the $A_{\infty}$-structure on $C_*(\Omega R)$ is necessarily formal just for degree reasons, which enables one to classify compact objects of this module category (ones whose endomorphism rings satisfy Poincar\'e duality) by direct algebraic means.

On the other hand, if $L \subset T^{2n}$ had vanishing homology class, it would lift as a closed aspherical Lagrangian submanifold of Maslov class zero to a covering space $X\rightarrow T^{2n}$ which is a \emph{subcritical} Stein manifold, namely that associated to $\pi_1(L)$.  However, no such $L$ can exist by results of Fukaya \cite{Fukaya}.  Briefly, Fukaya argues as follows.  Inside $X$ any compact set is Hamiltonian displaceable, in particular any Lagrangian bounds holomorphic discs; a suitable moduli space of such discs $\mathcal{M}$ with boundary in $L$ defines a Kuranishi chain for the free loop space $\mathcal{L}L$ of $L$, in degree determined by the Maslov class of $L$.  Crucially, the chain $\mathcal{M}$ determines a twisted homology $H_*(\mathcal{L}L;\mathcal{M})$ in which the cycle of constant loops $[L]$ becomes exact.  Properties of the string bracket for an aspherical manifold, together with this exactness, now imply that the Maslov class of $L$ could not in fact have vanished.
\end{proof}

\begin{cor} \label{cor:primitive}
If $L\subset (T^{2n},\omega_{st})$ is a Lagrangian torus of Maslov class zero, then $[L] \in H_n(T^{2n};\bZ)$ is a primitive class.
\end{cor}

For circles in $T^2$ the analogous result is a well-known consequence of the existence of geodesic representatives, and it would be interesting (but seems hard) to re-prove Corollary \ref{cor:primitive} using techniques of mean curvature flow.

\subsection{Basics} We now turn to the proof of Theorem \ref{thm:classify}.  Let $\Sigma$ be an embedded genus $2$ Lagrangian surface in $T^{4}$ of Maslov class zero, where the four-torus is equipped with its standard symplectic structure.

\begin{lem}
$\Sigma$ defines a non-zero object of $\sF(T^4)$. Moreover, it has endomorphism ring $HF(\Sigma,\Sigma) \cong H^*(\Sigma;\Lambda_{\bR})$.
\end{lem}

\begin{proof}
This follows immediately from Proposition \ref{Prop:SelfFloer} and  Lemma \ref{lem:well-defined}.
\end{proof}

\begin{lem} \label{lem:homologyoftori}
There are transverse linear Lagrangian tori $L_1$ and $L_2$ in $T^4$ with $[\Sigma]=[L_1]+[L_2] \in H_2(T^4;\bZ)$.
\end{lem}

\begin{proof}
The inclusion $\Sigma \hookrightarrow T^4$ induces an isomorphism on integral $H_1$, or $\Sigma$ would lift as a closed Lagrangian to a covering space $T^2 \times \bC^*$ of $T^4$ which contains no surfaces of non-zero square.  Therefore, the pullback map $H^1(T^4;\bZ) \rightarrow H^1(\Sigma;\bZ)$  is necessarily an isomorphism.  Pick a symplectic basis $\langle a_1,b_1,a_2,b_2\rangle$ for $H^1(\Sigma;\bZ)$ and denote by the same letters classes in $H^1(T^4)$ which project to these.  Note that $\omega|_{\Sigma} \equiv 0$ implies that $\langle a_1,b_1\rangle$ and $\langle a_2,b_2\rangle$  both span Lagrangian planes in the universal cover $\bR^4$ of $T^4$. The result follows.
\end{proof}

\begin{lem} \label{lem:pi2diesinH2}
The map $\pi_2(T^4,\Sigma_2) \rightarrow H_2(T^4,\Sigma_2)$ vanishes.  In particular, the isomorphism class of $\Sigma_2$ in $\sF(T^4)$ depends only on its Hamiltonian isotopy class.
\end{lem}

\begin{proof}
The long exact sequence 
\[
\pi_2(T^4) \rightarrow \pi_2(T^4,\Sigma_2) \rightarrow \pi_1(\Sigma_2) \rightarrow \pi_1(T^4)
\]
shows that $\pi_2(T^4,\Sigma_2)$ is the commutator subgroup of $\pi_1(\Sigma_2)$ which implies the first statement. The second statement is then an application of Lemma \ref{lem:Ham-independent}.
\end{proof}

\begin{lem} \label{lem:hits_fibres}
The Floer homology groups $HF(L_i,\Sigma)$ both have Euler characteristic $\pm 1$.
\end{lem}

\begin{proof}
Since $[\Sigma]^2=2=([L_1]+[L_2])^2$ and $[L_i]^2=0$, we see $[L_i]\cdot[\Sigma] = \pm 1$ (where the ambiguity of sign comes from not having fixed orientations on the surfaces).
\end{proof}

We now fix once and for all a Lagrangian fibration $\pi: T^4 \rightarrow T^2$ with Lagrangian fibre $L_f = L_1$ and section $L_s=L_2$, respectively.  That this is possible follows from the transitivity of $Sp_{4}(\bZ)$ on pairs of transverse linear Lagrangian 2-planes in $\bR^4$ defined over $\bZ$.

\subsection{The spectral sequence}
We now pass to the mirror, via Theorem \ref{thm:hms}, to deduce that there is a bounded complex $\cE_{\Sigma} = \cE_{\Sigma}^{\bullet} \in D^b(E\times E)$ of coherent sheaves with $\Ext^*(\cE_{\Sigma},\cE_{\Sigma}) \cong H^*(\Sigma_2;\Lambda_{\bR})$ as graded algebras.  In other words, $\cE_{\Sigma}$ is ``mirror" to the given Lagrangian genus 2 surface $\Sigma \subset T^4$. Moreover, we know that the rank of $\cE_{\Sigma}$ is $1$, as that rank is mirror to the Euler characteristic of $HF(\Sigma,L_f)$.  We recall:

\begin{lem} \label{lem:rkatleast2}
Let $A$ be an abelian surface over an algebraically closed field $k=\bar{k}$. Any coherent sheaf $\cP\rightarrow A$ satisfies $\rk_k\Ext^1(\cP,\cP) \geq 2$. 
\end{lem}

\begin{proof}
Heuristically, this holds since translations of $A$ define two infinitesimal deformations of $\cP$.    Formally, if $A'$ is an abelian surface and $V \rightarrow A'$ is locally free, 
\[ k^2 \cong H^1(A',\cO_{A'}) \subset H^1(A'; V^{\vee}\otimes V) = \Ext^1(V,V).
\]
If $\cP$ is an arbitrary sheaf on $A$, it has the same Ext-algebra as some vector bundle $V\rightarrow A'$  (tensor $\cP$ with a sufficiently ample line bundle and take the Fourier-Mukai dual), implying the result.  \end{proof}

We are heavily indebted to Dima Orlov for the proof of the next result.
\begin{lem}
Either $\cE_{\Sigma}$ is quasi-isomorphic to a sheaf, or $\cE_{\Sigma}$ is quasi-isomorphic to a length two complex $P^0 \rightarrow P^1$ of sheaves, with the cohomology sheaves $H^i$, $i=0,1$, of  $P^0\rightarrow P^1$ satisfying $\Ext^1(H^i,H^i) \cong H^1(T^2;\Lambda_{\bR})$.
\end{lem}

\begin{proof}
We use the spectral sequence
\begin{equation} \label{eqn:specseq}
E^2_{pq} = \bigoplus_{k-j=q} \Ext^p (H^j(\cE^{\bullet}),H^k(\cE^{\bullet})) \ \Longrightarrow \ \Ext^{p+q}(\cE^{\bullet}, \cE^{\bullet})
\end{equation}
where the $H^i(\cE^{\bullet}) = H^i$ denote the coherent cohomology sheaves of the complex $\cE^{\bullet}$.  Since we are working on an algebraic surface $E\times E$, the Ext-groups of any sheaf vanish in degrees $i\not \in \{0,1,2\}$, so the spectral sequence has 3 columns.  From the shape of the $E^2$-differential it follows that the ``middle box" $\oplus_j \Ext^1(H^j,H^j)$ survives to $E^{\infty}$.  Applying this to $\cE_{\Sigma}$, where $E^{\infty}$ is the associated graded group to a filtration on $H^*(\Sigma_2)$, the previous Lemma implies there are at most two cohomology sheaves.  In particular, $\cE_{\Sigma}^{\bullet}$ is quasi-isomorphic either to a sheaf or to a length 2 complex of sheaves.
Suppose we are in the latter situation.  Replace $\cE_{\Sigma} \simeq (P^0 \rightarrow P^1)$ by a length two complex.  The same argument (considering the ``box" surviving to $E^{\infty}$) shows that $\Ext^1(H^j, H^j)$ has rank 2 for $j\in \{0,1\}$.  
\end{proof}

\begin{lem} \label{lem:force-simple}
Suppose $\cE_{\Sigma} \simeq (P^0\rightarrow P^1)$ is quasi-isomorphic to a length two complex of coherent sheaves.  Let $H^i$ denote the cohomology sheaves of this complex. Then $\Ext^*(H^i, H^i) \cong H^*(T^2;\Lambda_{\bR})$ for $i\in \{0,1\}$.
\end{lem}

\begin{proof}
Since Serre duality identifies $\Ext^0(H^j,H^j)$ and $\Ext^2(H^j,H^j)$,   to complete the proof it is sufficient to show that the sheaves $H^j$ are simple.  Fix either one, and denote this by $G$. The coherent sheaf $G$ has a torsion filtration ($T_i(G)$ is the maximal subsheaf with $i$-dimensional support)
\begin{equation}
0 \subset  T_0(G) \subset T_1(G) \subset T_2(G) = G
\end{equation}
with $Q_i=T_i(G)/T_{i-1}(G)$ pure of dimension $i$.  Moreover, each of these pure subquotients $Q_i$ has a unique Harder-Narasimhan filtration
\begin{equation}
0 = HN_0(Q_i) \subset HN_1(Q_i) \subset \cdots \subset HN_l(Q_i) = Q_i
\end{equation}
such that the successive subquotients $R_j=HN_j(Q_i)/HN_{j-1}(Q_i)$ are semi-stable sheaves of pure dimension $i$.  The reduced or normalised Hilbert polynomials\footnote{i.e. the quotient of the usual Hilbert polynomials by their leading coefficients, cf. \cite[Chapter 1]{Huybrechts-Lehn}} satisfy $p(R_j) > p(R_k)$ if $j<k$, which implies $\Ext^0 (R_j,R_k)=0$ if $j<k$.  An analogous spectral sequence argument to that employed above then shows that the $\Ext^1$-groups of all semistable factors in these decompositions are subquotients of $\Ext^1(G,G)$ which has rank 2 by hypothesis, from which it follows that $G$ is of pure dimension and semistable.

Suppose for contradiction that $G$ is not simple.   Then $\Ext^0(G,G)$ has rank at least 2; by Serre duality this implies $\chi(G,G) > 0$. Any non-simple sheaf is not stable, hence has a Jordan-H\"older filtration
\begin{equation}
JH_0(G) \subset JH_1(G) \subset \cdots \subset JH_l(G) = G
\end{equation}
with the subquotients $JH_j/JH_{j-1}$ all stable sheaves of the same slope.  The associated graded object $gr(G)$ of this filtration is a sum of stable sheaves of the same slope $gr(G) = \oplus_{i=1}^l G_i$. Any non-zero morphism between stable sheaves of the same slope is an isomorphism, cf. p 55-56 of \cite{Simpson}, so the Euler characteristic $\chi(gr(G),gr(G))$ is a sum of terms $\chi(G_s, G_s) \leq 0$ (since a stable sheaf has $\Ext^0$ of rank 1 and $\Ext^1$ of rank at least 2) and $\chi(G_s, G_t) \leq 0$ (since $G_s \cong G_t$ or $\Ext^0(G_s,G_t)=0$). This gives a contradiction to $\chi(G,G)>0$. 
\end{proof}

\begin{cor}
In the case $\cE_{\Sigma}$ is a length two complex, with cohomology sheaves $H^0$ and $H^1$, then $\Ext^1(H^1,H^0)$ has rank one and $\cE_{\Sigma}$ is quasi-isomorphic to the associated extension. 
\end{cor}

\begin{proof}
The first statement follows from the spectral sequence \ref{eqn:specseq} and the fact that each of the $H^i$ has Ext-algebra the cohomology of a 2-torus.  The second statement is then automatic: the inclusion of the kernel $H^0$ and projection to the cokernel $H^1$ define a short exact sequence $H^0 \rightarrow (P^0 \rightarrow P^1) \rightarrow H^1$ which realises the desired quasi-isomorphism.
\end{proof}

The proof is now divided into cases, depending on the output of the previous discussion.

\subsection{$\cE_{\Sigma}$ is quasi-isomorphic to a length two complex} 

We use the following striking theorem of Mukai \cite[Proposition 4.11]{Mukai2}, proved using his theory of semi-homogeneous sheaves. Let $A$ be an abelian variety over an algebraically closed field $k=\bar{k}$ of characteristic zero.  Recall that a necessary and sufficient condition for a sheaf $P$ on $A$ to be semi-homogeneous is $\rk_{k}\Ext^1(P,P)=2$.

\begin{thm}[Mukai] \label{thm:Mukai}
Suppose $P^1$ and $P^2$ are simple semi-homogeneous sheaves and $\chi(P^1,P^2)=\pm 1$.  There is an abelian variety $A'$ defined over $k$, with a distinguished point $p\in A'$ (the origin in the group structure), an integer $m$, and a derived equivalence $D^b(A') \rightarrow D^b(A)$ under which $\cO_p \mapsto P^1$ and $\cO[m] \mapsto P^2$: on $A'$ the two sheaves arise as the skyscraper sheaf of the point $p$ and the structure sheaf (up to shift by $m$).
\end{thm}

Next we recall Orlov's theorem \cite{Orlov:abelian} that if $A$ and $A'$ are derived equivalent abelian varieties over $k$ then $A\times \widehat{A} \cong A' \times \widehat{A'}$, where $\widehat{\cdot}$ denotes the dual Abelian variety.  For us, $A=E\times E \cong \widehat{A}$ and we deduce that $A' \hookrightarrow E^4$ embeds in a self-product of the Tate curve.

\begin{lem}
The endomorphism ring $\End(E)$ is $\bZ$, i.e. the Tate curve does not have complex multiplication; its only self homomorphisms come from the group structure $(a \mapsto na)_{n\in\bZ}$.
\end{lem}

This is standard; for a proof see \cite[Chapter 5]{FresnelPut}.   Before continuing, recall that if $W$ is an abelian variety then a choice of polarisation on $W$ defines a canonical involution $\iota: \End_{\bQ}(W) \rightarrow \End_{\bQ}(W)$, called the Rosati involution.  Using the polarisation $H$ to identify $\phi_H: W \mapsto \widehat{W}$, the Rosati involution takes an endomorphism $f$ to $(\phi_H)^{-1}\widehat{f}\circ \phi_H$, where we use the fact that $\phi_H$ is invertible over $\bQ$.  We will fix a principal polarisation on $E$ coming from its description as a projective algebraic curve given in Definition \ref{defin:Tate}.

\begin{lem}
The only elliptic curve $E'$ which embeds in $E\times E$ is isomorphic to $E$ itself.
\end{lem}

\begin{proof}
According to a theorem of Birkenhake and Lange \cite{BL}, there is a bijective correspondence between abelian subvarieties of a polarised abelian variety $W$ and idempotents in the rational endomorphism algebra $\End_{\bQ}(W)$ symmetric with respect to the Rosati involution defined by the polarisation. Since $E$ is simple, if $W=E^2$ the rational endomorphism algebra can be canonically identified with the matrix algebra $M_2(\bQ)$ such that the Rosati involution takes a matrix to its transpose.  Idempotent elements have eigenvalues $0,1$ and the symmetry condition forces the eigenspaces to be orthogonal with respect to the standard diagonal inner product on $\bQ^2$.  For one-dimensional abelian subvarieties we exclude the idempotents $0$ and $1$.  We are therefore looking for elements of $M_2(\bQ)$ orthogonally conjugate to the matrices 
\[
\left( \begin{array}{cc} 1& 0 \\ 0 & 0\end{array} \right) \qquad  \left( \begin{array}{cc} 0& 0 \\ 0 & 1 \end{array} \right)
\]
 The  subvariety $A_{\epsilon}$ associated to an idempotent $\epsilon$ is given by the image of $n\epsilon$ for any $n\in\bZ$ sufficiently large that $n\epsilon$ is a genuine rather than rational endomorphism.  It follows that subvarieties associated to conjugate matrices are isomorphic, from which the result follows.
 \end{proof}
 
 \begin{rem} In this dimension, the symmetric idempotents are explicitly classified, giving the matrices:
\[
\frac{1}{1+r^2}\left( \begin{array}{cc} 1 & r \\ r & r^2 \end{array} \right) \qquad \frac{1}{1+r^2} \left( \begin{array}{cc} r^2 & -r \\ -r & 1 \end{array}\right)
\]
for $r=p/q \in \bQ$.   Clearing denominators, the subvarieties associated to these idempotents are just the images of the matrices $\left( \begin{array}{cc} q^2 & pq \\ pq & p^2 \end{array} \right)$ respectively $\left( \begin{array}{cc} p^2 & -pq \\ -pq & q^2 \end{array} \right)$.  Hence, these images are defined algebraically by the equations $\{(a,b) \in E\times E \, | \, pa= qb\}$ and $\{(a,b) \in E\times E \, | \, qa+pb = 0\}$ for $p/q \in \bQ$. Taking $p$ and $q$ coprime confirms directly that any such is the isomorphic image of $E$ under a homomorphism $x \mapsto (qx,  px)$ respectively $x \mapsto (px,-qx)$.
\end{rem}

\begin{lem}
The only abelian surface which embeds in $E^4$ is $E\times E$.
\end{lem}

\begin{proof}
Arguing as above, the result reduces to classifying the matrices with eigenvalues $0,1$ and eigenspaces orthogonal with respect to the standard form on $\bQ^4$.  These are all orthogonal conjugates in $M_4(\bQ)$ of the matrices whose only non-zero entries lie on the diagonal and all of which equal one.  The 2-dimensional subvarieties come from matrices conjugate to those with two non-zero diagonal entries.  The subvarieties associated to these diagonal matrices are obviously isomorphic to $E\times E$, hence so are those associated to the conjugates.
\end{proof}

In the case in which $\cE_{\Sigma}$ is quasi-isomorphic to a length two complex $(P^0 \rightarrow P^1)$, the proof is now completed as follows.  The Euler characteristics $\chi(H^i, H^i)$ for the two cohomology sheaves of the complex both vanish, so $\chi(\cE_{\Sigma}, \cE_{\Sigma}) = -2$ implies that $\{H^0, H^1\}$ satisfy the conditions of Mukai's theorem.  Moreover, since $\Ext^1(H^1,H^0)\neq 0$, the shift indeterminacy in that theorem is pinned down: $m=0$.  We quote from the literature:

\begin{enumerate}
\item Orlov's theorem \cite{Orlov:abelian} asserting that $Auteq(D^b(A))$ is a semi-direct product of the Hodge isometry group $U(A\times \widehat{A})$ with the subgroup $\bZ\times A\times \widehat{A}$ generated by shift, translations and tensoring by elements of $\Pic^0(A)=\widehat{A}$;
\item Polishchuk's identification $U(E^k\times\widehat{E^k}) \cong Sp_{2k}(\bZ)$ of the Hodge isometry group for self-products of an elliptic curve without complex multiplication \cite{Pol:weil}.
\end{enumerate}

We therefore see that some autoequivalence of the abelian surface $E\times E$  takes $\cE_{\Sigma}$ to the extension of  the structure sheaf by the skyscraper sheaf of a point.  Polishchuk's result implies that this autoequivalence can be realised geometrically by a symplectomorphism of the mirror torus $(T^4,\omega_{std})$.  It follows that, after applying a linear symplectomorphism and translation, the Lagrangian surface $\Sigma_2$ is isomorphic in the Fukaya category $\sF(T^4)$ to the cone defined by the intersection point of $L_f$ and $L_s$.  


\subsection{$\cE_{\Sigma}$ is quasi-isomorphic to a sheaf}

Having constructed our mirror map following Lemma \ref{lem:hits_fibres} so that $\cE_{\Sigma}$ has rank $1$, we consider two cases. Suppose first that $\cE_{\Sigma}$ is torsion-free, i.e. $T_0(\cE_{\Sigma}) = T_1(\cE_{\Sigma}) = 0$, so $\cE_{\Sigma}$ is pure of dimension 2.

\begin{lem}
If it is pure of dimension 2, $\cE_{\Sigma}$ is the kernel of a map from a line bundle of degree zero to  the skyscraper sheaf of a point.
\end{lem}

\begin{proof}
The double dual $(\cE_{\Sigma})^{\vee\vee}$ is reflexive, hence (on a surface) a line bundle, and there is a short exact sequence of sheaves

\begin{equation}
0 \rightarrow \cE_{\Sigma} \rightarrow (\cE_{\Sigma})^{\vee\vee} \rightarrow \cS \rightarrow 0
\end{equation}
where $\cS$ is of dimension zero, so is supported on a finite set.  For any point $p\in E\times E$ not in the support of $\cS$, since $(\cE_{\Sigma})^{\vee\vee}$ is locally free, the exact sequence of Ext's associated to this short exact sequence implies that $\Ext^*(\cE_{\Sigma}, \cO_p)$ is concentrated in degree 0.   Any such sheaf on a surface is $\cI_Z \otimes F$ for $\cI_Z$ the ideal sheaf of a zero-dimensional subscheme and $F$ a line bundle \cite{Huybrechts-Lehn}. Such a sheaf moves in a $2+2$length$(Z)$-dimensional family, where length$(Z) = \chi(\cO_Z, \cO_Z)$, so $\Ext^1(\cE_{\Sigma}, \cE_{\Sigma})$ being of rank 4 implies that $Z$ consists of a single reduced point.  Finally, the line bundle must have degree zero since we know $c_1(\cE_{\Sigma}) = 0$.
\end{proof}

Applying an autoequivalence of $D^b(E\times E)$ coming from tensoring by the inverse line bundle, we can assume without loss of generality that the line bundle in the conclusion of the previous Lemma is actually trivial.  It follows that $\cE_{\Sigma}$ is quasi-isomorphic to the length two complex $\cO \rightarrow \cO_p$;  translating back to the genus two surface, these sheaves are the mirrors of $L_s$ and $L_f$ respectively.

The final case to consider is where $\cE_{\Sigma}$ is a sheaf which is not pure of dimension 2, hence has a non-trivial torsion subsheaf $T_1(\cE_{\Sigma}) \subsetneq \cE_{\Sigma}$ yielding a short exact sequence
\begin{equation}\label{eqn:ses-use-mukai}
0 \rightarrow T_1(\cE_{\Sigma}) \rightarrow \cE_{\Sigma} \rightarrow \cS \rightarrow 0
\end{equation}
with $\cS$ torsion-free.  Rank considerations imply that  $\cS\neq 0$, so $\cE_{\Sigma}$ is defined by an extension class in $\Ext^1(\cS,T_1(\cE_{\Sigma}))$. 
The following result is due to Mukai \cite[Corollary 2.8]{Mukai:tata}, cf. also \cite[Lemma 2.7]{Huybrechts-Macri-Stellari}.

\begin{lem}
Let $0 \rightarrow A \rightarrow E \rightarrow B \rightarrow 0$ be an exact sequence of sheaves on a Calabi-Yau surface. If  $\Ext^0(A,B) = 0$  then 
\[
\rk_k \Ext^1(E,E) \geq \rk_k \Ext^1(A,A) + \rk_k \Ext^1 (B,B).
\]
\end{lem}

We apply this to Equation \ref{eqn:ses-use-mukai}.
We know $\Ext^0(T_1(\cE_{\Sigma}), \cS) =0$ since there are no morphisms from a torsion sheaf to a torsion-free sheaf. We deduce that 
\[
4=\rk_{\Lambda_{\bR}} \Ext^1(\cE_{\Sigma}, \cE_{\Sigma}) \geq \rk_{\Lambda_{\bR}} \Ext^1(\cS,\cS) + \rk_{\Lambda_{\bR}} \Ext^1(T_1(\cE_{\Sigma}),T_1(\cE_{\Sigma}))
\]
By Lemma \ref{lem:rkatleast2} we see that the sheaves $\cS$ and $T_1(\cE_{\Sigma})$ both have $\Ext^1$ of rank 2, hence are semi-homogeneous in the sense of Mukai.   One can argue as in Lemma \ref{lem:force-simple} to see that $\cS$ and $T_1(\cE_{\Sigma})$ are also both simple sheaves.  We have an Euler characteristic identity, recalling $\chi(\Sigma) = -2$ on the mirror side:
\[
-2 = \chi(\cE_{\Sigma},\cE_{\Sigma}) = \chi(\cS,\cS) + \chi(T_1(\cE_{\Sigma}), T_1(\cE_{\Sigma})) + 2\chi(\cS,T_1(\cE_{\Sigma}))
\]

From here, it is straightforward to see that $(\cS,T_1(\cE_{\Sigma}))$ satisfy the conditions of Mukai's Theorem \ref{thm:Mukai}, and one can apply the same steps used in the case in which $\cE_{\Sigma}$ was a length two complex (apply an autoequivalence to bring the relevant extension back to the case of a suitable line bundle and a skyscraper sheaf).  

Up until now we have derived algebraic information on $\cE_{\Sigma}$, whereas Theorem \ref{thm:classify} is formulated in terms of a Lagrange surgery.  Whether $\cE_{\Sigma}$ is quasi-isomorphic to a sheaf or a length two complex, the discussion so far reduces the proof of Theorem \ref{thm:classify} to the following Proposition, cf. Remark \ref{Rem:LagrangeSurgery}.

\begin{prop} \label{Prop:SurgeryIsCone}
The cone on the Floer cocycle defined by the unique intersection point of $L_f$ and $L_s$, graded in degree 1,  is quasi-represented by the graded Lagrange surgery of $L_f$ and $L_s$.
\end{prop}

\begin{proof}
The Lagrange surgery $\Sigma_{\#}$ of $L_f$ and $L_s$ \emph{is} a Lagrangian genus 2 surface of Maslov class zero, and so the preceding analysis of its own mirror complex $\cE_{\Sigma_{\#}}^{\bullet}$ shows that $\Sigma_{\#}$ is quasi-isomorphic to the cone on the morphism between a pair of transverse linear Lagrangian tori.  Moreover, these Lagrangian tori $\{R_1, R_2\}$ lie in the same homology classes as $\{L_f, L_s\}$, as determined by the homology class of the original surface $\Sigma_{\#}$ according to Lemma \ref{lem:homologyoftori}. Suppose for contradiction that 
\begin{equation} \label{eq:quasi}
\Sigma_{\#} \simeq \{ R_1 \rightarrow R_2\}
\end{equation} with $R_1\neq L_f$.  Then, choosing the support region for the Lagrange surgery to be sufficiently small, the geometric intersection of $\Sigma_{\#}$ with $R_1$ is a single transverse point, hence the Floer cohomology $HF(\Sigma_{\#}, R_1)$ has rank at most 1.  This contradicts the exact triangle
\[
HF^*(R_1,R_1) \rightarrow HF^*(R_2,R_1) \rightarrow HF^*(\{R_1 \rightarrow R_2\}, R_1)
\]
arising from the quasi-isomorphism of Equation \ref{eq:quasi}, since the first term is isomorphic to $H^*(T^2;\Lambda_{\bR})$ and the central term has rank 1.  The same argument shows that $R_2=L_s$, which proves the result.
\end{proof}

This completes the proof of Theorem \ref{thm:classify}.

\subsection{Application}  Theorem \ref{thm:classify} has the following more explicit implication for intersection properties of Lagrangian submanifolds.

\begin{cor} \label{cor:rankFloer}
If $\Sigma \subset T^4$ is a Lagrangian genus two surface, there are at least two linear Lagrangian tori $L_0, L_1\subset T^4$ with $\rk_{\Lambda_{\bR}}HF(L_i,\Sigma) \geq 3$. 
\end{cor}

\begin{proof}
The argument is precisely the same as for Proposition \ref{Prop:SurgeryIsCone}.
After applying a linear symplectomorphism, Theorem \ref{thm:classify} implies that we can take $\Sigma$ to be the Lagrange surgery of $L_f$ and $L_s$, in other words the cone of the unique morphism in the Fukaya category $L_f \rightarrow L_s$.  This gives an exact triangle of Floer cohomology groups
\[
HF^*(L_f, L_f) \rightarrow HF^*(L_f,L_s) \rightarrow HF^*(L_f,\Sigma) \rightarrow HF^{*+1}(L_f,L_f)
\]
The graded group $HF^*(L_f,L_s)$ has total rank one, since the two submanifolds meet transversely in a single point, whilst $HF(L_f,L_f) \cong H^*(T^2;\Lambda_{\bR})$ has rank 4.  It follows by exactness that $\rk HF(L_f,\Sigma) \geq 3$.  A similar argument shows that $\rk HF(L_s,\Sigma)$ satisfies the same bound.
\end{proof}

Recall that the Fukaya-isomorphism type of the Lagrangian genus two surface depends only on its Hamiltonian isotopy class, by Lemmas \ref{lem:Ham-independent} and \ref{lem:pi2diesinH2}.  Given this, Corollary \ref{cor:rankFloer} immediately implies Corollary \ref{cor:intersect} from the Introduction.


\appendix \label{appendix:speculation}
\section{Speculation}
Once the relevant foundational issues are clarified, one could apply the  generative criterion in this paper to a variety of higher-dimensional situations, some of which were mentioned in the Introduction, cf. Corollary \ref{cor:highdimension}.  In another direction, one could combine Theorem \ref{thm:generation} with (i) Seidel's proof of mirror symmetry for the quartic  $K3$ surface $Q\subset \bC\bP^3$ \cite{Seidel:HMSquartic} and (ii) Seidel's deformations-of-categories machinery for $\sF(Q)$ \cite{Seidel:talk}, to obtain the following evidence for a  ``refined Arnol'd conjecture":

\begin{prop}
Let $Z=Q\times T^2$ with the product symplectic form $\omega_{Q} \oplus \omega_{T^2}$.  Assume that $\sF(Z)$ is well-defined as a triangulated $A_{\infty}$-category over $\Lambda_{\bR}$.  Let $\phi \in \Symp(Z)$ act trivially on cohomology and have vanishing flux. Then Fix$(\phi) \neq \emptyset$, and if $\phi$ has non-degenerate fixed points it has at least $\sum_j b_j(Z)$ fixed points.
\end{prop}

The point here is that Seidel shows that the category $\scrF(Q)$ can be deformed to be essentially empty (perturb the symplectic form generically, mirror to a deformation of the quartic surface to a non-commutative $K3$ with no module which would be the analogue of a coherent sheaf).  The assumptions on $\phi$ imply that it deforms to an autoequivalence of the deformed category which is then necessarily the identity, and which has well-understood Floer cohomology.

If $\phi$ was assumed to be Hamiltonian isotopic to the identity this would be the content of the Arnol'd conjecture, but there may be ``exotic"  symplectomorphisms of $Q\times T^2$ which act trivially on cohomology and have vanishing flux but which are not isotopic to the identity (take an even power of a Dehn twist on the $K3$ factor stabilized by the identity).  Note that for $K3$ itself the analogous result holds by the Lefschetz fixed point theorem, whilst for tori $T^{2k}$ one can prove the result by the same  covering trick that proved Corollary \ref{cor:primitive} (consider lifting the graph $\Gamma(\phi)$ as a Lagrangian submanifold to the cotangent bundle $T^*\Delta$ of the diagonal $\Delta \subset T^{2k}\times T^{2k}$ and apply results on Lagrangian submanifolds of cotangent bundles obtained in \cite{Arnold:firststeps, FSS2}).  However, for $Q\times T^2$ there seems to be no elementary proof.


\end{document}